\documentclass[12pt]{article} 
\usepackage{lipsum}
\usepackage[utf8]{inputenc}
\usepackage{hyperref}
\usepackage[T1]{fontenc}
\usepackage[a4paper, margin=2.7cm]{geometry}
\usepackage{amsmath, amsthm, amsfonts, amssymb}
\usepackage[nameinlink]{cleveref}
\newtheorem{theorem}{Theorem}[section]
\newtheorem{Proposition}[theorem]{Proposition}
\newtheorem{Definition}[theorem]{Definition}
\newtheorem{Corollary}[theorem]{Corollary}
\usepackage{mathrsfs}           
\usepackage{bbm} 
\usepackage{breakurl}
\usepackage{url}
\usepackage{authblk} 
\theoremstyle{definition}
\theoremstyle{remark}
\theoremstyle{example} 
\usepackage{blindtext,titlefoot} 
\title{\bf On Decompositions of H-holomorphic\\ functions into quaternionic power series}
\author{Michael Parfenov}
\affil{\small \textit{Bashkortostan Branch of Russian Academy of Engineering, Ufa, Russia}} 
 \date{}
\AtEndDocument{\bigskip{\footnotesize\noindent}}
\begin{document}
\maketitle\unmarkedfntext{\!\!\!\!\!\!\!\!\!\!\!2020 Mathematics Subject Classification: 30G35. Date: July 31, 2024  \\e-mail address: mwparfenov@gmail.com as well as parfenm@gmx.de}
\begin{abstract}
Based on the full similarity in algebraic  properties and differentiation rules  between quaternionic (H-) holomorphic and complex (C-) holomorphic functions, we assume that there exists one holistic notion of a holomorphic function that has a H-representation in the case of quaternions and a  C-representation in the case of complex variables.  We get the essential  definitions and criteria for a quaternionic power series convergence, adapting  complex analogues to the quaternion case. It is established that the power series expansions of any holomorphic function in C- and H-representations are similar and  converge  with identical convergence radiuses.   We define a H-analytic function and prove that every H-holomorphic function is H-analytic.  Some examples of power series expansions are given.   \newline   
\end{abstract}
\section{Introduction}
A  possibility of a decomposition of  functions into  power series  is interesting in particular because it defines the notion of analytic functions  in classical mathematical analysis.   For spaces of dimension more than two,  for example, there are works devoted to quaternionic power series in the cases of regular  \cite{sud:qua} and slice-regular \cite{gs:pow}  quaternionic functions as well as in the more generalized case of Clifford-holomorphic \cite{ghs:fun} functions. We discuss here some aspects of power series expansions within the framework of the theory of the so-called H-holomorphic functions \cite{pm:eac}, \cite{pm:onno},  uniting the left and the right approaches. Now, we recall some results from this theory \cite{pm:eac}  needed for the sequel.

An independent quaternionic vaviable $p$ is further denoted by
$$p=x+yi+zj+uk,$$
and correspondingly a quaternion-valued function $\psi\left(p\right)$ by
$$\psi\left(p\right)=\psi_{1}\left(x,y,z,u\right)+\psi_{2}\left(x,y,z,u\right)i+\psi_{3}\left(x,y,z,u\right)j+\psi_{4}\left(x,y,z,u\right)k,$$
where $x,y,z,u$ and  $\psi_{1}\left(x,y,z,u\right),\psi_{2}\left(x,y,z,u\right),\psi_{3}\left(x,y,z,u\right),\psi_{4}\left(x,y,z,u\right)$ are real-valued, and $i,j,k$ are the base quaternions of the quaternion space $\mathbb{H}$ defined by the expressions
$$i^{2}=j^{2}=k^{2}=-1,$$
$$ij=-ji=k,\,\,\,\,jk=-kj=i,\,\,\,\,ki=-ik=j.$$

In the Cayley–Dickson construction (doubling form,  doubling procedure) \cite[p.~41]{ks:hn} we have
\begin{equation} \label{Eq1}
 \begin{gathered}
p=a+b\cdot j\in \mathbb{H},\;\;\;\;\;\;\psi\left(p\right)=\Phi_{1}\left(a,b,\overline{a}, \overline{b} \right)+\Phi_{2\cdot}\left(a,b,\overline{a}, \overline{b} \right)\cdot j \in \mathbb{H},  \\
\,\,\, \overline{ p}=\overline{ a}-b\cdot j\subset \mathbb{H}\text{,} \;\;\;\;\;\; \overline{ \psi\left(p\right)}=\overline{\Phi_{1}\left(a,b,\overline{a}, \overline{b} \right)}\;- \Phi_{2\cdot}\left(a,b,\overline{a}, \overline{b} \right)\cdot j \subset \mathbb{H} \text{,}
\end{gathered}
\end{equation}
where
\begin{equation}  \label{Eq2}
 \begin{gathered}
 a=x+yi,\;\;\;\;\;b=z+ui, \\ 
 \overline{a}=x-yi,\;\;\;\;\; \overline{b}=z-ui,
\end{gathered}
\end{equation}
$$\Phi_{1}\left(a,b,\overline{a}, \overline{b} \right)=\psi_{1}+\psi_{2}i, \;\;\;\;\;\Phi_{2}\left(a,b,\overline{a}, \overline{b} \right)=\psi_{3}+\psi_{4}i,$$
$$\overline{ \Phi_{1}\left(a,b,\overline{a}, \overline{b} \right)}=\psi_{1}-\psi_{2}i,\;\;\;\; \overline{ \Phi_{2}\left(a,b,\overline{a}, \overline{b} \right)}=\psi_{3}-\psi_{4}i$$
are compex quantities, the \textquotedblleft  $\cdot$\textquotedblright  \,and overbar signs denote, respectively, quaternionic multiplication and complex (or quaternionic, if needed) conjugation. For simplicity, we also use  the short designations $\Phi_{1}$ and $\Phi_{2},$ respectively, instead of $\Phi_{1}\left(a,b,\overline{a}, \overline{b} \right)$ and $\Phi_{2}\left(a,b,\overline{a}, \overline{b} \right)$. Since the quaternionic imaginary unit $j$ in the Cayley–Dickson doubling form should always be  recorded  behind the complex function $\Phi_{2}\left(a,b,\overline{a}, \overline{b} \right)$, to achieve this goal we often use the known \cite[\!\!p.~\!\!42]{ks:hn} relation $j\alpha=\overline{\alpha}j$, where $\alpha$ is any complex function.

Further, we use  the following necessary and sufficient conditions (\cite[\!p.~\!\!15]{pm:eac}, \cite[\!p.~\!\!4]{pm:onno}) for a function $\psi\left(p\right)$ to be H-holomorphic:
\begin{Definition}
It is assumed that the constituents $\Phi_{1}\left(a,b,\overline{a}, \overline{b} \right)$ and $\Phi_{2}\left(a,b,\overline{a}, \overline{b} \right)$ of a quaternionic function $\psi\left(p\right)=\Phi_{1}+\Phi_{2}\cdot j$ possess continuous first-order partial derivatives with respect to $a,\overline{a},b,$ and $\overline{b}$ in some open connected neighborhood $G_{4}\subset \mathbb{H}$ of a point $p\in \mathbb{H}$. Then a function $\psi\left(p\right)$ is said to be H-holomorphic (denoted by $\psi_{H}\!\left(p\right)$) at a point $p$ if and only if the functions $\Phi_{1}\left(a,b,\overline{a}, \overline{b} \right)$ and $\Phi_{2}\left(a,b,\overline{a}, \overline{b} \right)$ satisfy in $G_{4}$ the following quaternionic generalization of Cauchy-Riemann's equations:
\begin{equation} 
\label{Eq3} \left\{
\begin{aligned}
1)\,\,\,\,(\,\partial_{a}\Phi_{1}\!\!\mid \,\,=(\,\partial_{\overline{b}}\overline{\Phi_{2} }\! \mid, \,\,\,\,\,\,\,\,\,\,\,\,\,\,2)\,\,\,\,(\,\partial_{a}\Phi_{2}\!\!\mid &=-\,(\,\partial_{\overline{b}}\overline{\Phi_{1} }\! \mid 
,\\
3)\,\,\,\,(\,\partial_{a}\Phi_{1}\!\!\mid \,\,=(\,\partial_{b}\Phi_{2}\! \mid, \,\,\,\,\,\,\,\,\,\,\,\,\,\,4)\,\,\,\,(\,\partial_{\overline{a}}\Phi_{2}\!\!\mid &=-\,(\,\partial_{\overline{b}}\Phi_{1}\!\! \mid .
\end{aligned}
\right. 
\end{equation}
\end{Definition}
Here $\partial_{i},\,i=a,\overline{a},b,\overline{b},$ denotes the partial derivative with respect to $i$. The brackets $\left(\cdots{}\!\!\mid\right.$ with the closing vertical bar indicate that the transition $a=\overline{a}=x$ (to 3D space: $p=x+yi+zj+uk\rightarrow p_{3}=x+zj+uk$) has been already performed in expressions enclosed in brackets. Equations (\ref{Eq3}-1) and (\ref{Eq3}-2) have the components relating to the left quaternionic derivative \cite[\!p.~\!\!4]{pm:onno} and equations (\ref{Eq3}-3) and (\ref{Eq3}-4) to the right one. The requirement  $a=\overline{a}=x$ provides a possibility of joint implementation of equations (\ref{Eq3}-1,2) for the left quaternionic derivative and (\ref{Eq3}-3,4) for the right one. Equations (\ref{Eq3}-1) and (\ref{Eq3}-3) as well as (\ref{Eq3}-2) and (\ref{Eq3}-4) become, respectively, identical after the transition to 3D space \cite[\!\!p.~\!\!4]{pm:onno}, i.e. the left derivative becomes equal to the right one.

The H-holomorphy conditions (\ref{Eq3}) are defined so that during the check of the quaternionic holomorphy of any quaternionic function we have to do the transition  $a=\overline{a}=x$ in already calculated expressions for the partial derivatives of the functions $\Phi_{1}\left(a,b,\overline{a}, \overline{b} \right)$ and  $\Phi_{2}\left(a,b,\overline{a}, \overline{b} \right)$ and their complex conjugations.   Nevertheless, this doesn’t mean that we deal with triplets, since the transition $a=\overline{a}=x$ (or $y=0$) can not be initially done for quaternionic variables and functions. Otherwisewe we would lose a division operation. Any quaternionic function remains the same 4-dimensional quaternionic function regardless of whether we check its holomorphy or not. Simply put, the H-holomorphic functions are 4-dimensional quaternionic functions for which the partial derivatives of components of  the Cayley–Dickson doubling form satisfy equations (\ref{Eq3}) after the transition to 3D space. More clearly, they are those quaternionic functions whose the left and the right derivatives become equal after the transition to 3D space. 

Denoting the complex space (plane) by $\mathbb{C}$ and a complex holomorphic (C-holomorph-\\ic) function by $\psi_{C}\!\left(\xi\right)$, where $\xi$ is an independent complex variable, we have  the following
\begin{theorem} \label{th1.2}
Let a complex-valued function $\psi_{C}\!\left(\xi\right):G_{2}\rightarrow \mathbb{C}$ be C-holomorphic everywhere in an open connected  set $G_{2}\subseteq \mathbb{C}$, except, possibly, \!at certain singularities. Then a H-holomorphic function $\psi_{H}\!\left(p\right)$ of the same kind as $\psi_{C}\!\left(\xi\right)$ can be constructed (without change of a kind of function) from $\psi_{C}\!\left(\xi\right)$ by replacing a complex variable $\xi\in G_{2}$ in an expression for $\psi_{C}\!\left(\xi\right)$ by a quaternionic variable $p\in G_{4}\subseteq \mathbb{H}$, where $G_{4}$ is defined (except, possibly, at certain singularities) by the relation $G_{4}\supset G_{2}$ in the sense that $G_{2}$  exactly follows from $G_{4}$ upon transition from $p$ to $\xi$.  
\end{theorem}

We also need the following  \cite[\!p.~\!\!15]{pm:eac}
\begin{theorem} \label{th1.3}
Let a continuous quaternion-valued function $\psi_{H}\!\left(p\right)=\Phi_{1}+\Phi_{2}\cdot j,$
where $\Phi_{1}\left(a,b,\overline{a}, \overline{b} \right)$ and $\Phi_{2}\left(a,b,\overline{a}, \overline{b} \right)$ are differentiable with respect to $a,b,\overline{a}$ and $\overline{b}$, be  H-holomorphic everywhere in its domain of definition $G_{4}\subseteq \mathbb{H}$. Then its full quaternionic derivative (uniting the left and the right derivatives) with respect to $p$ defined by
\begin{equation}  
\label{Eq4} \psi_H^{'}\!\!\left(p\right)=\Phi_1^{'}+\Phi_2^{'}\cdot j,   
\end{equation}
where
\begin{equation}
\label{Eq5} \Phi_1^{'}= \partial_{a}\Phi_{1}+\partial_{\overline{a}}\Phi_{1},\,\,\,\,\Phi_{2}^{'}= \partial_{a}\Phi_{2}+\partial_{\overline{a}}\Phi_{2},
\end{equation}
is also H-holomorphic in $G_{4}.$
 If a quaternion function $\psi_{H}\!\left(p\right)$  is once quaternion-\\differentiable with respect to $p$ in $G_{4},$ then it possesses the full quaternionic derivatives of all orders in $G_{4},$ each one H-holomorphic. 
\end{theorem}
The generalized formula for the full $k'th$ quaternionic derivative \cite[\!\!p.~\!\!7]{pm:onno} is the following:
\begin{equation}
\label{Eq6} \psi_H^{\left(k\right)}\!\!\left(p\right)=\Phi_1^{(k)}+\Phi_2^{(k)}\cdot j,  
\end{equation}
where the constituents $\Phi_1^{(k)}$ and $\Phi_2^{(k)}$ are defined by
\begin{equation} 
 \label{Eq7}
\Phi_1^{(k)}=\partial_{a}\Phi_1^{(k-1)}\!\!+\partial_{\overline{a}}\Phi_1^{(k-1)},  \,\,\,\,\,\,\,\Phi_2^{(k)}=\partial_{a}\Phi_2^{(k-1)}\!\!+\partial_{\overline{a}}\Phi_2^{(k-1)}
\end{equation} 
and $\Phi_1^{(k-1)} \text{and}\,\, \Phi_2^{(k-1)}$ are the constituents of the $\left(k-1\right)^{'}\!\text{th}$ full derivative of  $\psi_{H}\left(p\right)$ represented in the Cayley–Dickson doubling form as $ \psi_H^{\left(k-1\right)}\!\!\left(p\right)=\Phi_1^{(k-1)}+\Phi_2^{(k-1)}\cdot j,  k\geq1;$ $\Phi_1^{(0)}=\Phi_{1}\!\!\left(a,b,\overline{a}, \overline{b} \right) \text{and}\,\, \Phi_2^{(0)}=\Phi_{2}\!\!\left(a,b,\overline{a}, \overline{b} \right)\!.$  

Theorem \ref{th1.3} leads  \cite[\!p.~\!\!15]{pm:eac} to the following   
\begin{Corollary} \label{co1.4} 
(a similarity between differentiating rules for C-holomorhic and H-holomorhic functions).\! All expressions for derivatives of a H-holomorphic function $\psi_{H}\!\left(p\right)$ of the same kind as a C-holomorphic analogue $\psi_{C}\!\left(\xi\right)$  have the same forms as the expressions for corresponding derivatives of a function $\psi_{C}\!\left(\xi\right).$
\end{Corollary}

For example, if the second derivative of the C-holomorphic function $\psi_{C}\!\left(\xi\right)=\xi^{k},\,\,k=1,2,3,\dots,$ is $\psi_C^{\left(2\right)}\!\left(\xi\right)=  k\left(k-1\right)\xi^{k-2},$  then the second full derivative of the H-holomorphic function $\psi_{H}\!\left(p\right)=p^{k}$ is $\psi_H^{\left(2\right)}\!\left(p\right) =k\left(k-1\right)p^{k-2}.$ Corollary \ref{co1.4} is illustrated by numerous examples of elementary quaternionic functions presented in \cite{pm:eac}, \cite{pm:onno}, \cite{pm:oph}.

Such a similarity is based on the above-mentioned uniting (see \cite[\!p.~\!\!18]{pm:oph}) the left and the right quaternionic derivatives, giving the full quaternionic derivative (\ref{Eq4}). Indeed, the component $\Phi_{2}^{'}$ in (\ref{Eq5}) consists of  the derivatives $\partial_{a}\Phi_{2}$ and $\partial_{\overline{a}}\Phi_{2}.$ The derivative $\partial_{a}\Phi_{2}$ belongs to equation (\ref{Eq3}-2) associated with the left quaternionic derivative, and the derivative $\partial_{\overline{a}}\Phi_{2}$ belongs to equation (\ref{Eq3}-4) associated with the right quaternionic derivative.  This also applies  to higher derivatives.
 
 \hspace{20mm}

\textit{The  goal of this article is to show that the  full similarity in algebraic properties and differentiation rules between the  C- and H-holomorphic functions, based on Theorems \ref{th1.2},  \ref{th1.3}, and Corollary \ref{co1.4} ,  can be extended to decompositions of H-holomorphic functions into infinite quaternionic power series.}   
\section{H-holomorphic linear combinations}
We need to establish, under what conditions a finite linear combination of H-holomorphic functions can also be H-holomorphic. We have first of all to consider the following
\begin{Proposition} \label{pr2.1} 
(multiplying by a constant). A H-holomorphic function $f_{H}\!\left(p\right)$ multiplied by a constant $r$: 
\begin{equation}
\label{Eq8} \psi_{H}\!\left(p\right)=rf_{H}\!\left(p\right) 
\end{equation} 
is also H-holomorphic only if a constant $r$ is real-valued.
Then the full $k'th$ order H-holomorphic derivative of a H-holomorphic function multiplied by a constant is the following: 
\begin{equation}
\label{Eq9} \psi_H^{(k)}\!\left(p\right) = \left(rf_{H}\!\left(p\right)\right)^{\left(k\right)}=r f_H^{\left(k\right)}\!\left(p\right).
\end{equation}
\end{Proposition}
\begin{proof}
Let a quaternionic function $f_{H}\!\left(p\right)=f_{1}\!\left(a,b,\overline{a},\overline{b} \right)+f_{2\cdot}\!\left(a,b,\overline{a},\overline{b} \right)\cdot j$ be H-holomorphic. Then it satisfies the conditions of H-holomorphy (\ref{Eq3}) as follows:
\begin{equation}
\label{Eq10}  \left\{
\begin{aligned}
1)\,\,\,\,(\,\partial_{a}f_{1}\!\!\mid \,\,=(\,\partial_{\overline{b}}\overline{f_{2} }\! \mid, \,\,\,\,\,\,\,\,\,\,\,\,\,\,2)\,\,\,\,(\,\partial_{a}f_{2}\!\!\mid &=-\,(\,\partial_{\overline{b}}\overline{f_{1} }\! \mid 
,\\
3)\,\,\,\,(\,\partial_{a}f_{1}\!\!\mid \,\,=(\,\partial_{b}f_{2}\! \mid, \,\,\,\,\,\,\,\,\,\,\,\,\,\,4)\,\,\,\,(\,\partial_{\overline{a}}f_{2}\!\!\mid &=-\,(\,\partial_{\overline{b}}f_{1}\!\! \mid .
\end{aligned}
\right. 
\end{equation}

As it has been found in \cite[\!p.\!~19]{pm:oph}, there exist in addition to main equations (\ref{Eq3}) for the H-holomorphic functions $\psi_{H}\left(p\right)=\Phi_{1}+\Phi_{2}\cdot j$ also the following equations:
\begin{equation}
\label{Eq11} \left\{ 
\begin{aligned}
a)\,\,\,\,\,\partial_{b}\Phi_{2} \!&=\,\partial_{\overline{b}}\overline{\Phi_{2}},\! \,\,\,\,\,\,\,\,\,\,\,\,\,\,b)\,\,\,\,\partial_{a}\Phi_{2}\!\! \,\,=-\,\partial_{\overline{b}}\Phi_{1},\! 
\\
c) \,\,\,\,\,\partial_{\overline{a}}\Phi_{1}\!&=\,\,\partial_{a}\overline{\Phi_{1} }, \!\,\,\,\,\,\,\,\,\,\,\,\,\,d)\,\,\,\,\partial_{\overline{a}} \Phi_{2} =-\,\,\partial_{\overline{b}}\overline{ \Phi_{1}},
\end{aligned}
\right.
\end{equation}
which hold true but already without having to carry out the transition to 3D space.
These equations are illustrated by all presented in \cite{pm:eac},  \cite{pm:onno} examples of H-holomorphic functions. 

Consider a quaternionic product $\psi\left(p\right)= C \cdot f_{H}\!\left(p\right)\!,$ where we initially suppose that the constant $C=C_{1}+C_{2}\cdot j$ is quaternion-valued, i.e. $C_{1}$ and $C_{2}$ are complex-valued constants. We need to define values of constants $C_{1}$ and $C_{2}$ for which the function $\psi\left(p\right)= C \cdot f_{H}\!\left(p\right)=\Phi_{1}+\Phi_{2}\cdot j$ is H-holomorphic. Let us recall the multiplication rule for quaternions in the Cayley–Dickson doubling form \cite[\!p.\!\!~42]{ks:hn} :
\begin{equation} \label{Eq12a}
 p\cdot q=\left(a_{1}+b_{1} \cdot j\right)\cdot\left(a_{2}+b_{2}\cdot j\right)=\left(a_{1}a_{2}-b_{1}\overline {b_{2}}\right)+\left(a_{1}b_{2}+\overline {a_{2}}\,b_{1}\right)\cdot j,
\end{equation} 
 where 
 \begin{align*}
&p=a_{1}+b_{1} \cdot j,\,\,\,\,\,\,\,\,q=a_{2}+b_{2} \cdot j,\\
&a_{1}=x_{1}+y_{1} i,\,\,\,\,\,\,\,\,\,\,b_{1}=z_{1}+u_{1} i,\\
&a_{2}=x_{2}+y_{2} i,\,\,\,\,\,\,\,\,\,\,b_{2}=z_{2}+u_{2} i.
\end{align*}

Using this rule, we obtain as follows:
\begin{equation*}
\psi\left(p\right)=C \cdot f_{H}\!\left(p\right) =\left(C_{1}+C_{2} \cdot j\right)\cdot\left(f_{1}+f_{2}\cdot j\right)=\left(C_{1}f_{1}-C_{2}\overline {f_{2}}\right)+\left(C_{1}f_{2}+C_{2}\overline {f_{1}}\right)\cdot j,
\end{equation*} 
whence
\begin{equation}
\label{Eq12} \Phi_{1}=C_{1}f_{1}-C_{2}\overline {f_{2}},\,\,\,\,\,\,\,\Phi_{2}=C_{1}f_{2}+C_{2}\overline {f_{1}}
\end{equation} 
and correspondingly
\begin{equation}
\label{Eq13} \overline{\Phi_{1}}=\overline{C_{1}}\, \overline{f_{1}}-\overline{C_{2}} f_{2},\,\,\,\,\,\,\,\overline{\Phi_{2}}=\overline{C_{1}}\,\overline{f_{2}}+\overline{C_{2}}f_{1}.
\end{equation}

To prove that a constant can be only a real value, it  suffices to use (\ref{Eq11}~-a). 
 Differentiating $\Phi_{2}$ with respect to $b$ and $\overline{\Phi_{2}}$ with respect to $\overline{b},$ and substituting the obtained derivatives into (\ref{Eq11}-a) we get the following equation: 
$$\partial_{b}\Phi_{2}=C_{1}\partial_{b}f_{2}+C_{2}\partial_{b}\overline{ f_{1}}=\partial_{\overline{b}}\overline{\Phi_{2}}=\overline{ C_{1}}\partial_{\overline{b}}\overline{f_{2}}+\overline{ C_{2}}\partial_{\overline{b}}f_{1},$$
 that must be satisfied if the function $\psi\left(p\right)= C \cdot f_{H}\!\left(p\right)=\Phi_{1}+\Phi_{2}\cdot j$ must be H-holomorphic. Since $f_{H}\!\left(p\right)$ is H-holomorphic, the equality $\partial_{b}f_{2}=\partial_{\overline{b}}\overline{f_{2}}$ holds true in this condition, and hence the equality $C_{1} \partial_{b}f_{2}=\overline{C_{1}} \partial_{\overline{b}}\overline{f_{2}}$ can only be satisfied if $C_{1}=\overline{C_{1}}.$

The derivative $\partial_{b}\overline{ f_{1}}$ does not belong to equations \!\! (\ref{Eq10}) when using them for the H-holomorphic function $f_{H}\!\left(p\right))=f_{1}+f_{2}\cdot j$, hence the condition $\partial_{b}\overline{ f_{1}}=\partial_{\overline{b}}f_{1}$ cannot be in general satisfied. In this case the equality $C_{2} \partial_{b}\overline{ f_{1}}=\overline{C_{2}} \partial_{\overline{b}}f_{1}$ can only be satisfied if $C_{2}=\overline{C_{2}}=0.$ 

It is not difficult to show that the function $\psi\left(p\right)= C \cdot f_{H}\!\left(p\right)$ satisfies all equations of system (\ref{Eq3}), if $C_{1}=\overline{C_{1}}$ and $C_{2}=0.$ \,For example, consider equation (\ref{Eq3}-2). If this equation  must be satisfied, then, differentiating the functions   $\Phi_{2}$ and $\overline{\Phi_{1}}$ defined, respectively, by (\ref{Eq12}) and (\ref{Eq13}),  we get the following equation:
\begin{equation*}
(\partial_{a}\Phi_{2}\!\!\mid=C_{1}(\partial_{a}f_{2}\!\!\mid+C_{2}(\partial_{a}\overline{f_{1}}\!\!\mid=-(\partial_{\overline{b}}\overline{\Phi_{1}}\!\!\mid=-\overline{C_{1}} (\partial_{\overline{b}}\overline{f_{1}}\!\!\mid+\overline{C_{2}}(\partial_{\overline{b}}f_{2}\!\!\mid,
\end{equation*}
which must be satisfied as well. Since, in accordance with equation (\ref{Eq10}-2), we have $(\partial_{a}f_{2}\!\!\mid=-(\partial_{\overline{b}}\overline{f_{1}}\!\!\mid,$ the partial equality $C_{1}(\partial_{a}f_{2}\!\!\mid=-\overline{C_{1}}(\partial_{\overline{b}}\overline{f_{1}}\!\!\mid$ will be satisfied only if $C_{1}= \overline{C_{1}}.$ On the other hand, since the function $(\partial_{a}\overline{f_{1}}\!\!\mid$ doesn't belong to system (\ref{Eq10}), the equality $(\partial_{a}\overline{f_{1}}\!\!\mid=(\partial_{\overline{b}}f_{2}\!\!\mid$ cannot be satisfied when using system  (\ref{Eq10}). We conclude that the equality $ C_{2}(\partial_{a}\overline{f_{1}}\!\!\!\mid= \overline{C_{2}}(\partial_{\overline{b}}f_{2}\!\!\!\mid$ can be satisfied only if $C_{2}=0.$

Thus, the function $\psi\left(p\right)=C\cdot f_{H}\!\left(p\right)$ satisfies equation  (\ref{Eq3}-2) only if $C_{1}=\overline{C_{1}}=r\in~\mathbb{R}$ and $C_{2}=0,$ i.e. $C=r\in\mathbb{R}.$ Quite analogously we can verify the validity of the remaining H-holomorphy equations of system  (\ref{Eq3}) if $C_{1}=\overline{C_{1}}=r\in\mathbb{R}$ and $C_{2}=0.$ The same results we obtain for H-holomorphic functions originally multiplied  on the right by a constant $C.$

We finally conclude that in order to get a H-holomorphic product of a H-holomorphic function by a constant, the constant can be only real-valued.  Using Corollary  \ref{co1.4}, further, we  prove expression (\ref{Eq9})\!. 
\end{proof}

\begin{Proposition} \label{pr2.2}
(H-holomorphic linear combinations) A sum of a finite number $n$ of H-holomorphic functions $f_{l}\!\left(p\right)$ each multiplied by a real constant $r_{l}:$
 \begin{equation}
\label{Eq14} \psi_{H}\!\left(p\right)=\sum_{l=0}^{n} r_{l}\,f_{l}\!\left(p\right)\!,
\end{equation}
where $\mid \!r_{1}\!\!\mid^{2}+\mid\! r_{2}\!\!\mid^{2}+\cdots +\mid\! r_{n}\mid^{2}\neq0,$
is H-holomorphic as well. The full $k'th$ order H-holomorphic derivative of $\psi_{H}\!\left(p\right)$ is the following:
 \begin{equation*}
\psi_H^{(k)}\!\left(p\right)= \left(\sum_{l=0}^{n} r_{l}\,f_{l}\!\left(p\right)\right)^{\!\!\!\left(k\right)}\!=\sum_{l=0}^{n} r_{l}\,f_l^{\left(k\right)}\left(p\right).
\end{equation*}
\end{Proposition}
\begin{proof}
It follows from Proposition \ref{pr2.1},  linearity of equations (\ref{Eq3}) and Corollary  \ref{co1.4}\,. 
\end{proof}

\section{Power  series expansions of H-holomorphic functions} 
The algebraic properties of  addition, subtraction, multiplication, and division of H-holomorphic functions are fully identical (isomorphic) with complex ones.  This is based on the established facts that the quaternionic multiplication of the H-holomorphic functions behaves as commutative (see the proof in  \cite[\!p.\!\!~18]{pm:eac}) and the left quotient of them is equal to the right one  (see the proof in  \cite[p.20]{pm:eac}).
As a result of algebraic operations with H-holomorphic functions we also get H-holomorphic functions \cite[\!p.\!\!~17]{pm:eac}.

It is evident that due to these facts,  the commutative and associative laws for addition and multiplication, and the distributive laws  in the case of the  H-holomorphic functions are the sames as in the case of  C-holomorphic functions.   Note, however, that this  only applies to holomorphic functions $\psi_{H}\!\left(p\right)$ and $\psi_{C}\!\left(\xi\right)$. \textit{Thus, one can speak of the isomorphism between the algebras of H-holomorphic and C-holomorphic functions.} The differentiation rules for  H-holomorphic functions are also fully similar to complex ones \cite[\!p.\!\!~17]{pm:eac} and results of differentiating are also H-holomorphic.

Formula (\ref{Eq12a})  for quaternionic mulniplication  gives as a particular case  the formula  for complex multiplication  (see \cite[\!p.\!~18]{pm:eac}) after the transition to the complex case, i.e. at $a_{1}=x_{1}\left(y_{1}=0\right),b_{1}=z_{1}\left(u_{1}=0\right),a_{2}=\overline{a_{2}}=x_{2}\left(y_{2}=0\right),
b_{2}=\overline{b_{2}}=z_{2}\left(u_{2}=0\right)$.  Quite analogously  we  get the complex Cauchy-Riemann equations as a particular case  \cite[\!p.\!~8]{pm:onno} of the  necessary and sufficient conditions (\ref{Eq3}) for a quaternionic function to be H-holomorphic.  Given these facts, we can formulate the following
\begin{Proposition} \label{pr2.3}
We assume that there exists in principle one holistic  notion of a holomorphic function that in the case of quaternions is represented by a H-holomorphic function and in the case of complex variables by a C-holomorphic function.  Replacement of a complex variable  $\xi$ by a quaternionic $p$  carries out in accordance with Theorem \ref{th1.2} the transition from the complex representation of a holomorphic function in a set $G_{2}$ to its quaternionic representation  in a set $G_{4}$. Correspondingly, we can speak of a  C-representation and  a H-representation of one  holistic notion of holomorphic functions. 
\end{Proposition}

 According to Theorem \ref{th1.2}, the quaternionic power functions $p^{l},\,l=0,1,2,\dots,$ and their derivatives of all orders are H-holomorphic. Putting $f_{l}\left(p\right)=p^{l}$ in (\ref{Eq14}), we obtain the following H-holomorphic polynomial:
 \begin{equation}\label{Eq15} 
\psi_{H}\left(p\right)=S_{n}\left(p\right)=\sum_{l=0}^{n}r_{l}p^{l},\,n=1,2,3,\dots,\,\,\,\,\,\,r_{l}\in\mathbb{R},\,p\in\mathbb{H}.
\end{equation}

It is natural to introduce an object that is like a H-holomorphic polynomial, but with infinitely many terms. Continuing (\ref{Eq15}) by increasing the $n,$ we get the following quaternionic power series:
 \begin{equation} \label{Eq16} 
\psi\left(p\right)=S\left(p\right)=\sum_{l=0}^{\infty}r_{l}p^{l}=r_{0}+r_{1}p^{1}+r_{2}p^{2}+\cdots+r_{l}p^{l}+\cdots,\,\,\,\,\,\,r_{l}\in\mathbb{R},\,p\in\mathbb{H},  
\end{equation}
which is a Taylor series $ \sum_{l=0}^{\infty}r_{l}\left(p-p_{0}\right)^{l}$ at $ p_{0}=0$ or a Maclaurin  series \cite[\!p.\!\!~353]{pm:hm}  for the function $\psi\left(p\right)$. 

Since all algebraic operations and differentiating rules for H-holomorphic power functions are identical with those in complex analysis, we can introduce notions related to quaternionic power series  similarly to complex ones.  Given Proposition \ref{pr2.3}, it is not unreasonable to expect that if a C-representation of a holomorphic function by a convergent power series exists, then a H-representation of this holomorphic function by a convergent power series must also exist.  However, we would like to verify this directly, and therefore we need first to adapt some complex definitions and convergence criteria to the quaternionic case.  This is also important because we want to show a consistency between complex and quaternionic power series.
 
 \textit{As in complex analysis} \cite{pm:hm}, \textit{a notion of an infinite quaternionic series $\sum_{l=0}^{\infty}r_{l}p^{l}$ is based on a notion of a convergent sequence $\left(S_{n}\left(p\right)\right)$ of partial sums $S_{n}\left(p\right)=\sum_{k=0}^n r_{k}p^{k},\\n=0,1,2,\dots$}.

By analogy with \cite[\!p.\!\!~99]{mh:ca} we introduce the following
\begin{Definition}
\textit{An infinite quaternionic series  $\sum_{l=0}^{\infty}r_{l}p^{l}$ is said to be convergent if the sequence $\left(S_{n}\left(p\right)\right)$ of the partial sums $S_{n}\left(p\right)=\sum_{l=0}^{n} r_{l}\,p^{l}$ is convergent. At that, $\lim_{n \rightarrow {\infty}}
S_{n}\left(p\right)=S\left(p\right)$ is said to be a sum $S\left(p\right)$ of this infinite series and $r_{l}p^{l}$ its term. Otherwise the series is divergent}. A power series $\sum_{l=0}^{\infty}r_{l}p^{l}$ is said to be \textit{absolutely convergent} if a series $\sum_{l=0}^{\infty}\mid \!\!r_{l}\,p^{l}\!\!\mid\,=\!\sum_{l=0}^{\infty} \!r_{l}\! \mid \!\!p\!\!\mid^{l},$  where $\mid \!\!p\!\!\mid=\sqrt{x^{2}+y^{2}+z^{2}+u^{2}}=\sqrt{a\overline{a}+b\overline{b}},$ converges. A power series $\sum_{l=0}^{\infty}r_{l}p^{l}$ is said to be  \textit{conditionally convergent} if it converges but not absolutely.  
\end{Definition}

Analogously to complex analysis \cite[\!p.\!\!~209]{mh:ca} we also introduce the following
\begin{Definition}
A quaternionic series $\sum_{l=0}^{\infty}r_{l}p^{l}$ is said to be uniformly convergent to its sum $S\left(p\right)$ on the set $T\!\subset\mathbb{H}$ if it converges at all $p\in T$ to  $S\left(p\right)$, and for every $\epsilon>0$ there exists an integer $N_{\epsilon}>0$ (that depends only on $\epsilon$) such that if $n\geq N_{\epsilon}$, then $\mid S_{n}\left(p\right)-S\left(p\right)\mid<\epsilon$ \textit{for all} $p\in T.$ 
\end{Definition}
We generalize the complex definition of Cauchy criterion \cite[\!p.\!\!~250]{pm:hm}  to the quaternionic case as follows.
\begin{Definition} \label{def3.3}
(a quaternionic Cauchy criterion)
A quaternionic series $\sum_{l=0}^{\infty} r_{l}\,p^{l}$ converges uniformly to $S\left(p\right)$ on $T$  if and only if for every $\epsilon >0$ there is an integer $N_{\epsilon}>0$ such that $\mid \!S_{n}\left(p\right)-S_{m}\left(p\right)\!\mid <\epsilon$ for all $n,m >N_{\epsilon}$ and for  \textit{all} $p\in T.$
\end{Definition}
We associate every quaternionic power series with a real number $R\geq0$ called  \textit{a radius of convergence of a series} \cite[\!p.\!\!~110]{mh:ca}. We further define a quaternionic generalization of a complex domain of convergence as  an open connected ball $B\left(0,R\right)=\left\{p\!:\mid \!p\mid<R \right\}\subset\mathbb{H}$ centered at $p=0$ such that $\psi\left(p\right):B\left(0,R\right)\rightarrow \mathbb{H},$ where $R$ is a radius of convergrnce. For simplicity, we do not consider functions with  singularities. 

Further, we reproduce some assertions and theorems from complex analysis \cite{mh:ca}, \cite{sh:ica}, \cite{pm:hm} as \textquotedblleft propositions\textquotedblright \,  adapted to quaternions. At that, we repeat complex proofs, adapting them to the quaternionic case. We begin by the following important property of convergent series.
\begin{Proposition} \label{pr3.4}  
(a term test). The only series $\sum_{l=0}^{\infty} r_{l}\,p^{l}$ that can converge are those ones whose terms approach $0.$ In other words, if $\sum_{l=0}^{\infty} r_{l}\,p^{l}$ converges, then $r_{l}\,p^{l}\rightarrow 0$ as $l\rightarrow{\infty}.$
\end{Proposition}
\begin{proof}
If the series  $\sum_{l=0}^{\infty} r_{l}\,p^{l}$ converges in $B\left(0,R\right)$, then the limit of the sequence of its partial sums approaches the sum $S$, i.e. $S_{n}\rightarrow S$ as $n\rightarrow {\infty}$, where $S_{n}$ is the $n'th$ partial sum $S_{n}=\sum_{k=0}^{n} r_{k}\,p^{k}$.  Putting $n=l$, we have $$\lim_{l \rightarrow \infty}r_{l}\,p^{l}=\lim_{l \rightarrow \infty}\left(S_{l}-S_{l-1}\right)=\lim_{l \rightarrow \infty}S_{l}-\lim_{l \rightarrow \infty}S_{l-1}=S-S=0.$$ 
\end{proof}
The contrapositive of that statement gives a test which can tell us that some series diverge. Note, however, the terms converging to $0$ don't imply the series converges. The term test is only necessary, but not a sufficient convergence test.

Analogously to \cite[\!p.\!\!~105]{mh:ca} we introduce the following
\begin{Proposition} \label{pr3.5}
(d'Alembert's Ratio Test). If the series $\sum_{l=0}^{\infty} r_{l}\,p^{l}$ is a quaternionic power series with the property that 
\begin{equation} \label{Eq18d}
\lim_{l \rightarrow \infty}\frac{\mid\! r_{\left(l+1\right)}\,p^{{l+1}} \!\mid }{\mid\! r_{l}\,p^{l} \!\mid}=L
\end{equation} 
(provided the limit exists), then the series is absolutely convergent in an open connected ball $B\left(0,\frac{1}{L}\right)$ if $L<1$ and divergent if $L>1$. If $L=1$, then the test is inconclusive, so we have to use some other test. If the series is convergent, then $R=\frac{1}{L}$ (possibly infinite) is a radius of convergence.
\end{Proposition}
\begin{proof}
Since we only deal with positive values in (\ref{Eq18d}), the proof of this proposition is quite analogous to complex one \cite[\!p.\!~193]{ap:tm}. We skip the details.
\end{proof}

The Weierstrass M-test, given in \cite[\!p.\!\!~210]{mh:ca}, can be generalized to the quaternionic case as follows.
\begin{Proposition} \label{pr3.6} 
(a quaternionic Weierstrass M-test). 
Suppose the infinite quaternionic series $\sum_{l=0}^{\infty} r_{l}\,p^{l}$ has the property 
\begin{equation*} 
\mid \!\!r_{l}\,p^{l} \!\!\mid \leq M_{l},\,\,\,M_{l}\geq0
\end{equation*}
for each $l\in\left\{0,1,2, \dots\right\}$ and all $p\in T.$  If the series $\sum_{l=0}^{\infty} M_{l}$ converges, then the series $\sum_{l=0}^{\infty} r_{l}\,p^{l}$ converges uniformly and absolutely on $T.$
\end{Proposition}
\begin{proof}
Consider the sequence $\left\{S_{n}\left(p\right)\right\}$  of partial sums $S_{n}\left(p\right)=\sum_{l=0}^{n} r_{l}\,p^{l}$. If $n>m,$ where $n,m\in\mathbb{N}$, then $\mid~\!\!\!S_{n}\left(p\right)-S_{m}\left(p\right)\!\!\mid= \mid \!\!\!\!\!~\!\sum_{l=m+1}^{n}r_{l}\,p^{l}\!\!\mid=\sum_{l=m+1}^{n}\!\!\mid \!\!r_{l}\,p^{l} \!\!\mid \leq \sum_{l=m+1}^{n}M_{l}.$ Since the series $\sum_{l=0}^{\infty}M_{l}$ converges, the last expression can be made as small as we wish by choosing the positive integer $m$ large enough. Thus for a given value $\epsilon>0$ there is a positive integer  $N_{\epsilon}$ such that if $n,m > N_{\epsilon},$ then $\mid~\!\!\!S_{n}\left(p\right)-S_{m}\left(p\right)\!\!\mid < \epsilon.$  But this means that for all  $p\in T,$ according to Definition \ref{def3.3}, the series $\sum_{l=0}^{\infty} r_{l}\,p^{l}$ converges uniformly on a set $T$. Since the series $\sum_{l=0}^{\infty}\mid \! r_{l}\,p^{l}\!\mid$ is majorized by the convergent series $\sum_{l=0}^{\infty}M_{l}$, the series $\sum_{l=0}^{\infty} r_{l}\,p^{l}$  also converges absolutely on $T$.
\end{proof}
\begin{Proposition} \label{pr3.7} 
(terms bounded by a constant) If terms of a quaternionic power series $\sum_{l=0}^{\infty} r_{l}\,p^{l}$ are bounded at some point $p_{0}\in\mathbb{H}$:
 \begin{equation} \label{Eq17}
 \mid \!r_{l}p_0^l \mid\leq M,
\end{equation}
where $M>0$ is a constant and $l=0,1,2, \dots$, then this series converges in an open connected ball $B\left(0,\mid\! p_{0} \!\mid\right)=\left\{p:\,\mid \!p \!\mid < \mid p_{0} \!\mid \right\}$. Moreover, it converges absolutely and uniformly on any compact subset $K$ that is properly contained in $B\left(0,\mid\! p_{0} \mid\right)$.
\end{Proposition}
\begin{proof}
We adapt  the proof considered in  \cite[\!p.\!~96]{sh:ica} to the quaternionic case. \!Suppose that $p_{0}\neq 0$ and \! $\mid \!p_{0}\!\!\mid=\zeta>0$. Otherwise the ball $B\left(0,\zeta\right)$ is empty. Let $K\in B\left(0,\zeta\right)=\left\{p:\,\mid \!p \mid <\zeta\right\}$, then there exists $q<1$ such that $\frac{\mid p\mid}{\zeta}\leq q<1$ for all $p\in K$. Therefore \textit{for any} $p\in K$ and \textit{any}  of the above $l$ we have $r_{l}\!\!\mid \!p^{l}\mid\leq r_{l}\,\zeta^{l}q^{l}$. However, inequality (\ref{Eq17}) implies that $\mid \!r_{l}p_0^l \mid=r_{l}\zeta^{l}\leq M$, then $r_{l}\!\!\mid \!p^{l}\mid\leq Mq^{l}$ and the series $\sum_{l=0}^{\infty} r_{l}\,p^{l}$ is majorized by the convergent series $M\sum_{l=0}^{\infty} q^{l}$ (the geometric series $\sum_{l=0}^{\infty} q^{l}\!\!, \,q<1$, is convergent) \textit{for all} $p\in K$. Therefore, according to Proposition \ref{pr3.6}, the series $\sum_{l=0}^{\infty} r_{l}\,p^{l}$ converges uniformly and absolutely on $K$. This proves the second statement of this proposition. The first statement follows from the second, since any given point $p^{'}\in B\left(0,\zeta \right)$ is contained in some ball $B\left(0,\zeta^{'} \right)=\left\{p\!:\mid \!p\mid<\zeta^{'}\right\}\!,\,\,\mid \!p^{'}\!\!\mid<\zeta^{'}<\zeta,$ belonging compactly to $B\left(0,\zeta \right)=B\left(0,\mid \!p_{0} \!\mid \right)$.
\end{proof}

The quaternionic sets $G_{4}$ can be regarded as compact, since  we assume that they are continuous and bounded; every sequence in them has a subsequence that converges to an element (a point)  contained again in them.

In complex analysis the following theorem (Weierstrass) \cite[\!p.\!\!~106]{sh:ica} holds true: \\
\textit{If the functional series}
 \begin{equation}  \label{Eq18a}
f\left(\xi\right)=\sum_{l=0}^{\infty}f_{l}\left(\xi\right)\!,\,\,\,\,\xi\in \mathbb{C}
\end{equation}
\textit{of  C-holomorphic functions} $f_{l}$ \,\textit{in an open connected domain}\, $D \left(\xi\right)$ \textit{converges  uniformly on any compact subset of this domain, then  the sum} $f\left(\xi\right)$ \textit{of this series is holomorphic in $D\left(\xi\right)$ and  the series can be differentiated termwise:}
 \begin{equation} \label{Eq19} 
f^{'}\left(\xi\right)=\sum_{l=0}^{\infty}f_l^{'}\left(\xi\right).
\end{equation}
\textit{Such a differentiation can be performed arbitrarily many times at any point} $\xi\in D\left(\xi\right)$.

Without loss of generality, we regard a domain $D\left(\xi\right)$ as a set $G_{2}$ (Theorem \ref{th1.2}) that is an open connected disk $D\left(0,R\right)=\left\{\xi:\mid \xi \mid<R \right\}$, where $R>0$ is a radius of convergence of the series (\ref{Eq18a}). Correspondingly, we also regard an open connected quaternionic ball $B\left(0,R\right)=\left\{p\!:\mid p\mid<R \right\}$ as a quaternionic generalization $G_{4}$ of a complex disk $D\left(0,R\right)$ such that $D\left(0,R\right)$ follows from $B\left(0,R\right)$ upon the transition from $p$ to $\xi$. 

We see that the requirements for the termwise differentiation in complex analysis are the following: (A) the functions $f_{l}\left(\xi\right)$ are C-holomorphic in $D\left(\xi\right)$; (B)  the series (\ref{Eq18a}) converges uniformly on any compact subset of $D\left(\xi\right)$. At that, we can compare radiuses of convergence of the series before and after differentiation. 
If analogous requirements will be fulfilled in the case of quaternionic functions and series, then we can apply Theorem \ref{th1.3} and Corollary \ref{co1.4} to prove the possibility of a termwise differentiability in the quaternionic case. Adapting all this to the quaternionic case we formulate  the following
\begin{Proposition} \label{pr3.8} 
(a termwise differentiation)
If the quaternionic power series
 \begin{equation}  \label{Eq18} 
f\!\left(p\right)=\sum_{l=0}^{\infty} r_{l}\,p^{l},\,\,\,\,p\in \mathbb{H},\,\,\,l=0,1,2,\dots
\end{equation}
of the H-holomorphic functions $r_{l}\,p^{l}$ in its open connected convergence ball $B\left(0,R\right)=\left\{p\!:\mid \!p\!\mid<R \right\}$ converges uniformly on any compact subset $T\in B\left(0,R\right)\!,$ then the sum $f\!\left(p\right)$ of this series is H-holomorphic in $B\left(0,R\right)$ and the series can be termwise differentiated at each point $p \in B\left(0,R\right)$:
\begin{equation} \label{Eq21} 
f^{'}\!\left(p\right)=\sum_{l=0}^{\infty}\left(r_{l}\,p^{l}\right)^{'}=\sum_{l=0}^{\infty}l\,r_{l}\,p^{l-1}=r_{1}+2r_{2}p+3r_{3}p^{2}+\cdots + lr_{l}p^{l-1}+\cdots .
\end{equation} 
Such a differentiation can be performed arbitrarily many times at any point $ p \in B\left(0,R\right)$.
\end{Proposition}
\begin{proof}
It is evident, according to Theorems \ref{th1.2}\,, \ref{th1.3} and Corollary \ref{co1.4}\,,  that the function $f_{l}\left(p\right)= r_{l}\,p^{l}$ and its first derivative $\left( r_{l}p^{l}\right)^{'}=lr_{l}\,p^{l-1}$ (also derivatives of all orders) are H-holomorphic in $B\left(0,R\right)$. Hence the requirement (A) of complex analysis is fulfilled in the quaternionic case. The H-holomorphy of $f\!\left(p\right)=\sum_{l=0}^{\infty} r_{l}\,p^{l}$ in a convergence ball $B\left(0,R\right)$ follows from Theorem \ref{th1.2} applied to the  C-holomorphic function $\psi_{C}\!\left(\xi\right)=\sum_{l=0}^{\infty} r_{l}\,\xi^{l}$ in its convergence disk $D\left(0,R\right)$. We now prove that $f\!\left(p\right)=\sum_{l=0}^{\infty} r_{l}p^{l}$ converges uniformly on any compact subset $T\subset B\left(0,R\right)$. Take any $p\in B\left(0,R\right)$ and consider any positive $\zeta$ such that $\mid \! p\! \mid< \zeta=\mid\! p_{0}\!\mid<R$. By virtue of the convergence of the series
\begin{equation*}
\sum_{l=0}^{\infty} r_{l}\,\zeta^{l}=r_{0}+r_{1}\zeta+r_{2}\zeta^{2}+\cdots + r_{l}\zeta^{l}+\cdots,\,\,\,\,\,r_{l},\zeta\in\mathbb{R},
\end{equation*} 
and in accordance with  Proposition \ref{pr3.4} we have $r_{l}\,\zeta^{l}\rightarrow 0$ as $l\rightarrow{\infty}.$ Then there exists a big enough constant $M>0$ such that the general term of this series is bounded above:
\begin{equation*}
r_{l}\zeta^{l}\leq M,\,\,\,l=0,1,2,\dots. 
\end{equation*}
 We have $r_{l}\!\! \mid \!\!p_{0} \!\!\mid^{l}=r_{l}\zeta^{l}\leq M$ and series (\ref{Eq18}) converges absolutely and uniformly on any set $T\subset B\left(0,\zeta\right)$ in accordance with Proposition \ref{pr3.7}. Hence the requirement (B) of the complex analysis is also fulfilled in the quaternionic case.

Now we compare radiuses of convergence of the series before and after differentiation. For an absolute value of the $l^{\,'}\!th$ term of the series (\ref{Eq21}) we have the following estimate:
\begin{equation*}
\mid\! \left(r_{l}\,p^{l} \right)^{'}\!\!\mid=lr_{l}\mid \!p^{l-1}\!\mid=lr_{l}\zeta^{l}\cdot\mid\!\frac{p}{\zeta}\!\mid^{l-1} \cdot \frac{1}{\zeta}\leq \frac{M}{\zeta}\cdot l\cdot \mid\!\frac{p}{\zeta}\!\mid^{l-1},\,\,\,\,\,\,l=0,1,2,\dots,
\end{equation*}
where the expression for the first derivative: $\left(r_{l}\,p^{l} \right)^{'}\!\!=l r_{l}p^{l-1}$ is used. The series
\begin{equation*}
\sum_{l=0}^{\infty}\mid\!\left(r_{l}\,p^{l} \right)^{'}\!\!\mid\leq\frac{M}{\zeta}\sum_{l=0}^{\infty}l \mid\!\frac{p}{\zeta}\!\mid^{l-1}=\frac{M}{\zeta}\left\{1+2\mid\!\frac{p}{\zeta}\!\mid +3\mid\!\frac{p}{\zeta}\!\mid ^{2}+\cdots +l\mid\!\frac{p}{\zeta}\!\mid ^{l-1}+\cdots \right\}
\end{equation*}
is absolutely convergent in $ B\left(0,\zeta\right)$. Given $\mid\frac{p}{\zeta}\mid <1$, one can verify this fact by using  Proposition  \ref{pr3.5}\,. Then, by Proposition \ref{pr3.6}, the series $\sum_{l=0}^{\infty}\left(r_{l}\,p^{l} \right)^{'}$ converges absolutely and uniformly on $B\left(0,\zeta\right)$. Since the series $\sum_{l=0}^{\infty}\left(r_{l}\,p^{l} \right)^{'}$ converges absolutely in $B\left(0,\zeta\right)$ and the choice of $p_{0}$ is arbitrary, it is clear that the convergence radius $R^{'}$ of the series (\ref{Eq21}) is not less than the convergence radius $R$ of the series (\ref{Eq18}). This conclusion is sufficient for our objective and  we don't need to prove that $R^{'}$ is exactly equal to $R$. By repeating this reasoning , we prove that the series (\ref{Eq18}) can be  differentiated termwise arbitrarily many times in its ball of convergence.    
\end{proof}

\begin{Proposition} \label{pr3.9}
(a quaternionic Maclaurin series). Let a H-holomorphic function $\psi_{H}\!\left(p\right)$ be represented by a sum of a convergent quaternionic power series:
 \begin{equation*}
\psi_{H}\!\left(p\right)=\sum_{k=0}^{\infty} r_{k}\,p^{k}=r_{0}+r_{1}p+r_{2}p^{2}+\cdots + r_{k}p^{k}+\cdots,\,\,\,\,r_{k}\in \mathbb{R}
\end{equation*}
in a ball $B\left(0,R\right)=\left\{p\!:\mid \!p\!\mid<R \right\}$.Then the coefficients $r_{k}$ are determined uniquely as
\begin{equation} \label{Eq22} 
r_{k}=\frac{\psi_H^{\left(k\right)}\!\!\left(0\right)}{k!},\,\,\,\,k=0,1,2,\dots,
\end{equation} 
where $\psi_H^{\left(k\right)}\!\!\left(0\right)\!, k=1,2,3, \dots$ is the $k^{'}th$ full derivative of $\psi_{H}\!\left(p\right)$ at $p=0$ and $\psi_H^{\left(0\right)}\!\!\left(0\right)$ is $\psi_{H}\!\left(0\right)$. 
\end{Proposition}
\begin{proof}
Since $\psi_{H}\!\left(p\right)$ is H-holomorphic in a ball $B\left(0,R\right)$, there exist, according to Theorem \ref{th1.2}, all derivatives $\psi_H^{\left(k\right)}\!\!\left(0\right),\, k=1,2, \dots,$ in $B\left(0,R\right)$. Putting $p=0$ in $\psi_{H}\!\left(p\right)= \sum_{k=0}^{\infty} r_{k}\,p^{k}$ we find $r_{0}=\psi_{H}\!\left(0\right)$. Differentiating  $\psi_{H}\!\left(p\right)= \sum_{k=0}^{\infty} r_{k}\,p^{k}$ termwise and using the formula $\left(r_{k}\,p^{k} \right)^{'}= r_{k}k p^{k-1}$, we obtain 
\begin{equation*}
\psi_H^{'}\!\left(p\right) =r_{1}+2r_{2}p+3r_{3}p^{2}+\cdots + kr_{k}p^{k-1}+\cdots.
\end{equation*}

Putting $p=0$ into this expression, we obtain $r_{1}=\psi_H^{'}\!\left(0\right)$. Further, differentiating the expression for $\psi_H^{'}\!\left(p\right)$, we get the second derivative as follows: 
\begin{equation*}
 \psi_H^{\left(2\right)}\!\left(p\right)=2r_{2}+3\cdot2 r_{3}p+4\cdot3r_{4}p^{2}+5\cdot4r_{5}p^{3}+\cdots+ k\cdot\left(k-1\right)r_{k}p^{k-2}+\cdots.
\end{equation*}
Putting $p=0$ into this expression, we obtain $r_{2}=\frac{\psi_H^{\left(2\right)}\!\left(0\right)}{2}.$ Differentiating $ \psi_{H}\!\left(p\right)=\sum_{k=0}^{\infty} r_{k}\,p^{k}\,\,k$ times, we get $$\psi_H^{\left(k\right)}\!\left(p\right) =k! r_{k}+positive\, powers\,of\,p\,multiplied \,by\,integers.$$
Then once again putting $p=0$, we obtain $r_{k}=\frac{\psi_H^{\left(k\right)}\!\!\left(0\right)}{k\,!}.$
\end{proof}

Note that representations of H-holomorphic functions by Maclaurin series are unique, since the coefficients of series are determined uniquely by formula (\ref{Eq22}) when using the full quaternionic derivatives (\ref{Eq6}) of the functions in question. Therefore, in a problem of representing a H-holomorphic function by a quaternionic power series the answer does not depend on the methods adopted for this purpose.

We can introduce the notion of a H-analytic function by the following
\begin{Definition}
(a H-analytic function)  A quaternionic function is said to be   H-analytic in an open connected ball $B\left(0,R\right)$ if it can be expanded  as a convergent  Maclaurin (Taylor) series in $B\left(0,R\right)$, where $R\geq 0$ is a radius of convergence. 
\end{Definition}
Whether we speak of a Maclaurin or a Taylor series, doesn't matter here, since we can go from one series to another by replacing variables. 
Now  we can formulate the following
\begin{Proposition} \label{pr3.13} 
Every H-holomorphic function is H-analytic. 
\end{Proposition}
\begin{proof}
According to Proposition  \ref{pr3.9},  every H-holomorphic function can be uniquely expanded as a convergent  Maclaurin (Taylor) series in some open connected ball $B\left(0,R\right)$, where $R\geq 0$ is a radius of series convergence. 
\end{proof}
\begin{Proposition} \label{pr3.10} 
(a similarity between C- and H-representations of power series). The power series expansions of every holomorphic function in C- and H-representations are similar and  converge  with identical convergence radiuses $R$. In a C-representation we regard an $R$ as a radius of a convergence disk and in a  H-representation as a radius of a corresponding convergence ball.
\end{Proposition}
\begin{proof}
According to Theorem \ref{th1.2} , every H-holomorphic function (also its series in H-representation) defined in an open connected domain $G_{4}\subseteq \mathbb{H}$ is constructed  from a C-holomorphic function of the same kind (also its series in C-representation) defined in an open connected domain $G_{2}\subseteq \mathbb{C}$ by replacing a complex variable by a quaternionic one, and $G_{4}$ is such that $G_{2}$  exactly follows from $G_{4}$ upon the transition from $p$ to $\xi$.  Given this relationship, if we regard, as mentioned above, a domain $G_{2}$ as a complex disk $D\left(0,R\right)=\left\{\xi:\mid \xi \mid<R \right\}$, then a domain $G_{4}$ can only be a quaternionic  ball $B\left(0,R\right)=\left\{p\!:\mid \!p\mid<R \right\}$ with the same $R$, i.e. the convergence  radiuses of the series of a holomorphic function in C- and H-representstions are equal. The similarity of series (the same coefficients of the terms of the same degrees) in C- and H-representations also follows from Theorem \ref{th1.2}.
\end{proof}

The quaternionic generalization of Laplace's equations within the theory of H-holomorphic functions is presented in \cite[\!p.\!~5]{pm:qpc}.  At that, H-harmonic components of H-holomorphic functions are denoted  by $\psi_{1}\left(x,y,z,u\right),\,\psi_{2}\left(x,y,z,u\right),\,\psi_{3}\left(x,y,z,u\right)$ and $\psi_{4}\left(x,y,z,u\right).$ 

Note that if coefficients of an initial complex series contain explicitly the imaginary unit $i$, we must use replacing \cite[\!p.\!~6]{pm:onno} the imaginary complex unit $i$ with the quaternionic imaginary unit $r$ in order to obtain a quaternionic analogue of this complex series. 
\section{Examples} 
\subsection{Examples of the elementary functions \mathversion{bold}$e^{p},\,\sin p, \cos p$}  \label{sub4.1}
In complex analysis we have, for example, the following power series expansions \cite{mh:ca}:
\begin{align*} 
&e^{\xi}=\sum_{l=0}^{\infty}\frac{\xi^{l}}{l\,!}=1+\xi+\frac{\xi^{2}}{2\,!}+\frac{\xi^{3}}{3\,!}+\cdots,\,\,\,\xi \in \mathbb{C}, \\  & \sin \xi=\sum_{l=0}^{\infty}{\left(-1\right)}^{l} \frac{\xi^{2l+1}}{\left(2l+1\right)!}=\xi-\frac{\xi^{3}}{3!}+\frac{\xi^{5}}{5!}-\frac{\xi^{7}}{7!}+\cdots,\\ &  \cos \xi=\sum_{l=0}^{\infty}{\left(-1\right)}^{l} \frac{\xi^{2l}}{\left(2l\right)!}=1-\frac{\xi^{2}}{2!}+\frac{\xi^{4}}{4!}-\frac{\xi^{6}}{6!}+\cdots.
\end{align*}

Replacing $\xi$ by $p$ in these expressions, we obtain in accordance with Theorem \ref{th1.2} the following power series expansions about a point $p=0$ for the H-holomorphic functions $e^{p},\,\sin p,$ and $\cos p:$

\begin{equation} \label{Eq23} 
 \hspace*{-1.3cm} e^{p}=\sum_{l=0}^{\infty}\frac{p^{l}}{l!}=1+p+\frac{p^{2}}{2\,!}+\frac{p^{3}}{3\,!}+\cdots,\,\,\,p \in \mathbb{H}, 
\end{equation}
\begin{equation}  \label{Eq24}
\sin p=\sum_{l=0}^{\infty}{\left(-1\right)}^{l} \frac{p^{2l+1}}{\left(2l+1\right)!}=p-\frac{p^{3}}{3!}+\frac{p^{5}}{5!}-\frac{p^{7}}{7!}+\cdots, 
\end{equation}
\begin{equation} \label{Eq25} 
 \hspace*{-0.70cm} \cos p=\sum_{l=0}^{\infty} {\left(-1\right)}^{l} \frac{p^{2l}}{\left(2l\right)!}=1-\frac{p^{2}}{2!}+\frac{p^{4}}{4!}-\frac{p^{6}}{6!}+\cdots. 
\end{equation}
The verification of the H-holomorphy of the functions $\sin p $ and $\cos p$ on the left-hand of these expressions can be made analogously to those of H-holomorphy of the function $e^{p}\!,$ which is given in \cite[p.\!\!~6]{pm:onno}. We skip the details.

In accordanse with Proposition \ref{pr2.3} we can speak of a H-representation of the functions $exp, sin$ and $cos.$ 

Applying d'Alembert's ratio test (Proposition \ref{pr3.5}) to the series $e^{p}=\sum_{l=0}^{\infty} r_{l}\,p^{l},$ where $r_{l}=\frac{1}{l!},$ we have $$L=\lim_{l \rightarrow \infty}\frac{\mid\! r_{\left(l+1\right)}\,p^{{l+1}} \!\mid}{\mid \!r_{l}\,p^{l} \!\mid}=\lim_{l \rightarrow \infty}\frac{ \mid p^{{l+1}} \!\mid l\,!}{\left(l+1\right)! \mid \!p^{l} \!\mid }= \mid \!p\!\mid \!\!\lim_{l \rightarrow \infty}\frac{1}{l+1}=0<1.$$
Then, in accordance with this test, the radius of convergence of the series $e^{p}=\sum_{l=0}^{\infty}\frac{p^{l}}{l\,!}$ is $R=\frac{1}{L}=\infty.$

Analogously, for the series $\sin p=\sum_{l=0}^{\infty}{\left(-1\right)}^{l} \frac{p^{{2l+1}}}{\left(2l+1\right)!}$ we find 
\begin{align*} 
L&=\lim_{l \rightarrow \infty}\frac{\mid\! r_{\left(l+1\right)}\,p^{{l+1}} \!\mid}{\mid \!r_{l}\,p^{l} \!\mid}=
\lim_{l \rightarrow \infty}\frac{\left(-1\right)^{l+1} \mid p^{2\left(l+1\right)+1} \!\mid \left(2l+1\right)! }{\left[2\left(l+1\right)+1\right]!\,\left(-1\right)^{l} \mid \!p^{{2l+1}} \!\mid } \\  
 &=-\lim_{l \rightarrow \infty}\frac{\mid p^{2}\!\mid}{\left(2l+2\right)\left(2l+3\right)}=0<1,
\end{align*}
whence the radius of convergence of the power series expansion of $\sin p$ is also $R=\frac{1}{L}=\infty.$ The calculation for $\cos p$ gives $R=\infty$ as well. Thus, the convergence set of series expansions for the functions $e^{p}\!,\,\sin p,\cos p$ is the set of all quaternion numbers. We see that these results correspond  Proposition \ref{pr3.10} and we further can immediately get a convergence radiuses in similar cases without making any calculations.

The verification of the uniform convergence of these series is analogous to those used in the proof of Proposition \ref{pr3.8}, since these series can be majorized for any $p\in\mathbb{H}$ by the convergent series with the terms $r_{l}\,\zeta^{l},$  where $\zeta$ is a positive number such that $\mid \! p \!\mid <\zeta <\infty.$ According to Proposition \ref{pr3.8} , we can differentiate termwise these series at all points $p\in\mathbb{H}.$ At that, we get the following expressions for derivatives:

\begin{align*} 
&\left(e^{p}\right)^{'}=\left(\sum_{l=0}^{\infty}\frac{p^{l}}{l!}\right)^{'}=\sum_{l=0}^{\infty}\left(\frac{p^{l}}{l!}\right)^{'}=\sum_{k=1}^{\infty}\frac{p^{k-1}}{\left(k-1\right)!}=\sum_{l=0}^{\infty}\frac{p^{l}}{l!}=e^{p},\\  
&\left(\sin p\right)^{'}=\left(\sum_{l=0}^{\infty}{\left(-1\right)}^{l} \frac{p^{2l+1}}{\left(2l+1\right)!}\right)^{'}=\sum_{l=0}^{\infty}{\left(-1\right)}^{l}\left(\frac{p^{2l+1}}{\left(2l+1\right)!}\right)^{'}\\
&=\sum_{l=0}^{\infty}{\left(-1\right)}^{l}\frac{p^{2l}}{\left(2l\right)!}=\cos p,\\
&\left(\cos p\right)^{'}=\left(\sum_{l=0}^{\infty}{\left(-1\right)}^{l} \frac{p^{2l}}{\left(2l\right)!}\right)^{'}=\sum_{l=0}^{\infty}{\left(-1\right)}^{l}\left(\frac{p^{2l}}{\left(2l\right)!}\right)^{'}\\
&=\sum_{k=1}^{\infty}{\left(-1\right)}^{k}\frac{p^{2k-1}}{\left(2k-1\right)!}=-\sum_{l=0}^{\infty}{\left(-1\right)}^{l}\frac{p^{{2l+1}}}{\left(2l+1\right)!}=-\sin p,
\end{align*}
which are analogous to  those in complex analysis.

Now we want to show that formulae (\ref{Eq23}), (\ref{Eq24}) and (\ref{Eq25}) can be obtained by means of Maclaurin series expansions (Proposition \ref{pr3.9}) of the H-holomorphic functions $e^{p},\,\sin p,$ and $\cos p.$ In other words, we want to show that the coefficients of these series, which are preserved during the transition from the complex to the quaternion case by the replacement of a complex variable by a quaternionic one, are coefficients (\ref{Eq22}) of Maclaurin series expansions of these functions.

We can represent the quaternion variable $p=x+yi+zj+uk$ as a sum of real and imaginary parts: $p=x+Vr,$ where $V=\sqrt{y^{2}+z^{2}+u^{2}}$ is a real value and $$r=\frac{yi+zj+uk}{\sqrt{y^{2}+z^{2}+u^{2}}}$$
is a purely imaginary unit quaternion, so its square is $-1.$ Since $r^{2}=-1$ as well as $x$ and $V$ are real values, the quaternionic formula $p=x+Vr$ is algebraically equivalent to the complex formula $\xi=x+yi.$ Then we generalize the complex Euler formula \cite{mh:ca}  as $e^{{rV}}=\cos V+r \sin V,$ and  then the known complex formulae $\cos \xi=\frac{e^{{i\xi}}+e^{{-i\xi}}}{2}$ and $\sin \xi=\frac{e^{{i\xi}}-e^{{-i\xi}}}{{2i}}$ as :
\begin{equation}  \label{Eq26} 
\cos p=\frac{e^{{rp}}+e^{{-rp}}}{2}, 
\end{equation}
\begin{equation}  \label{Eq27}
\sin p=\frac{e^{{rp}}-e^{{-rp}}}{{2r}},
\end{equation}
where the complex imaginary unit $i$ is replaced  \cite[\!p.\!~6]{pm:onno} by the quaternionic $r$.

Using (\ref{Eq26}), (\ref{Eq27}) and relations $x=\frac{a+\overline{a}}{2},\,\,y=\frac{a-\overline{a}}{{2i}},\,\,z=\frac{b+\overline{b}}{2},\,\,u=\frac{b-\overline{b}}{{2i}},$ following from (\ref{Eq2}), we obtain after some cumbersome calculation the following expressions for the functions $\cos p$ and $\sin p$ in the Cayley–Dickson doubling form:
\begin{equation*} 
 \cos p=\frac{e^{{rp}}+e^{{-rp}}}{2}=\Phi_{{1\left(\cos p\right)}}+\Phi_{{2\left(\cos p\right)}}\cdot j,
\end{equation*}
where
\begin{equation} \label{Eq28} 
\Phi_{{1\left(\cos p\right)}}=\frac{\left(e^{{-V}}+e^{V}\right)\cos \frac{a+\overline{a}}{2}}{2}+\frac{\left(a-\overline{a}\right)\left(e^{{-V}}-e^{V}\right)\sin \frac{a+\overline{a}}{2}}{{4V}},
\end{equation}
\begin{equation} \label{Eq29} 
\hspace*{-5.2cm} \Phi_{{2\left(\cos p\right)}}=\frac{\left(e^{{-V}}-e^{V}\right)\sin \frac{a+\overline{a}}{2}}{{2V}}b
\end{equation}
and
\begin{equation*}
\sin p=\frac{e^{{rp}}-e^{{-rp}}}{{2r}}=\Phi_{{1\left(\sin p\right)}}+\Phi_{{2\left(\sin p\right)}}\cdot j,
\end{equation*}
where (see also \cite[p.~19]{pm:eac})
\begin{equation} \label{Eq30} 
\Phi_{{1\left(\sin p\right)}}=\frac{\left(e^{{-V}}+e^{V}\right)\sin \frac{a+\overline{a}}{2}}{2}-\frac{\left(a-\overline{a}\right)\left(e^{{-V}}-e^{V}\right)\cos \frac{a+\overline{a}}{2}}{{4V}},
\end{equation}
\begin{equation} \label{Eq31} 
\hspace*{-4.65cm} \Phi_{{2\left(\sin p\right)}}=-\frac{\left(e^{{-V}}-e^{V}\right)\cos \frac{a+\overline{a}}{2}}{{2V}}b.
\end{equation}

According to \cite[\!p.\!~5]{pm:onno}, we have  the function $e^{p}=\Phi_{{1\left(e^{p}\right)}}+\Phi_{{2\left(e^{p}\right)}}\cdot j$, where
\begin{equation} \label{Eq32} 
\Phi_{{1\left(e^{p}\right)}}=2\beta\cos V+\frac{\beta\left(a-\overline{a}\right)\sin V}{V},
\end{equation}
\begin{equation} \label{Eq33} 
\hspace*{-2.9cm} \Phi_{{2\left(e^{p}\right)}}=\frac{2\beta\sin V}{V}b,
\end{equation}
and $\beta=\frac{e^{\frac{a+\overline{a}}{2}}}{2}.$

Now we calculate the values of the functions $e^{p},\,\sin p,$ and $\cos p$ as well as the value of $V=\sqrt{y^{2}+z^{2}+u^{2}}$ at the point $p=a+b\cdot j=0,$ i.e. at
\begin{equation} \label{Eq34}
a=0,\,\,\,\,\,\,\overline{a}=0,\,\,\,\,\,\,b=0,\,\,\,\,\,\,\overline{b}=0\,\,\left(\text{or}\,\, x=y=z=u=0\right),   \,\,\,\,\,\,V=0.
\end{equation}
At that, to eliminate indeterminate expressions of the form $\frac{0}{0}$ we use the known  limit relation \cite[\!p.\!\!~233]{mh:ca}:
\begin{equation}   \label{Eq35} 
\lim_{V \rightarrow 0}\frac{\sin V}{V}=1
\end{equation}
as well as the following limit relation calculated by using L’Hospital’s rule \cite[\!p.\!\!~254]{pm:hm}:
\begin{equation} \label{Eq36} 
\lim_{V \rightarrow 0}=\frac{e^{{-V}}-e^{V}}{V}=\lim_{V \rightarrow 0}\frac{\left(e^{{-V}}-e^{V}\right)^{'}}{V^{'}}=\lim_{V \rightarrow 0}\frac{{-e}^{{-V}}-e^{V}}{1}=-2.
\end{equation}

Substituting (\ref{Eq34}), (\ref{Eq35}), (\ref{Eq36}) into (\ref{Eq28}), (\ref{Eq29}),  (\ref{Eq30}),  (\ref{Eq31}), (\ref{Eq32}) and (\ref{Eq33}), we get the following values for the functions $e^{p},\,\sin p,$ and $\cos p$ at the point $p=0:$
\begin{equation}  \label{Eq37}
e^{0}=1,\,\,\,\,\sin 0=0,\,\,\,\,\cos 0=1,
\end{equation}
coinciding with analogous values of the corresponding functions in real and complex areas.   Further, using (\ref{Eq37}), we get the expansion coefficients for series (\ref{Eq18}) when $l=0$ as follows: $$r_{0\left(e^{p}\right)}=e^{0}=1,\,\,\,\,r_{0\left(\sin p\right)}=\sin 0=0,\,\,\,\,r_{0\left(\cos p\right)}=\cos 0=1.$$
They coincide with the coefficients of the terms of degree 0 in (\ref{Eq23}), (\ref{Eq24}), and (\ref{Eq25}).

As shown in \cite[\!p.\!\!~7]{pm:onno}, the derivatives of all orders of the quaternionic function $e^{p}$ are equal to the function itself. Then, according to (\ref{Eq22}), we get $r_{k\left(e^{p}\right)}=\frac{\psi_H^{\left(k\right)}\!\!\left(0\right)}{k\,!}=\frac{\left(e^{p}\right)^{{\left(k\right)}}\left(0\right)}{{k!}}=\frac{e^{0}}{{k!}}=\frac{1}{{k!}},$ where $k=0,1,2,\dots.$ The obtained values of coefficients coincide with the corresponding values of coefficients in (\ref{Eq23}). The Maclaurin series expansion for the function $e^{p}$  coincides with series (\ref{Eq23}) obtained by using Theorem \ref{th1.2}.

Now, we show how the derivatives of $\cos p$ and $\sin p$ are further computed within the framework of H-holomorphic functions theory. According to (\ref{Eq4}), we define the first full quaternionic derivative of $\cos p$ as follows:
\begin{equation} \label{Eq37a} 
\left(\cos p\right)^{'}=\Phi_{{1\left(\cos p\right)}}^{'}+\Phi_{{2\left(\cos p\right)}}^{'}\cdot j,
\end{equation}
where
\begin{equation} \label{Eq38} 
\Phi_{{1\left(\cos p\right)}}^{'}=\partial_{a}\Phi_{{1\left(\cos p\right)}}+\partial_{\overline{a}}\Phi_{{1\left(\cos p\right)}},
\end{equation}
\begin{equation} \label{Eq39} 
\Phi_{{2\left(\cos p\right)}}^{'}=\partial_{a}\Phi_{{2\left(\cos p\right)}}+\partial_{\overline{a}}\Phi_{{2\left(\cos p\right)}}.
\end{equation}

Now we need to compute the partial derivatives of $\Phi_{{1\left(\cos p\right)}}$ and $\Phi_{{2\left(\cos p\right)}}$ with respect to each of $a,\overline{a}.$ After some cumbersome calculations, we have for $\partial_{a}\Phi_{{1\left(\cos p\right)}}$ and $\partial_{\overline{a}}\Phi_{{1\left(\cos p\right)}}$ the following expressions:
\[ \begin{split}
\partial_{a}\Phi_{{1\left(\cos p\right)}}=\partial_{a}&\left[\frac{\left(e^{{-V}}+e^{V}\right)\cos \frac{a+\overline{a}}{2}}{2}+\frac{\left(a-\overline{a}\right)\left(e^{{-V}}-e^{V}\right)\sin \frac{a+\overline{a}}{2}}{{4V}}\right] \\
 = \frac{1}{2}&\left[\frac{\left(a-\overline{a}\right)\left(e^{{-V}}-e^{V}\right)\cos \frac{a+\overline{a}}{2}}{{4V}}-\frac{\left(e^{{-V}}+e^{V}\right)\sin \frac{a+\overline{a}}{2}}{2}\right]\\
+\frac{1}{2}&\left[\frac{\left(e^{{-V}}-e^{V}\right)\sin \frac{a+\overline{a}}{2}}{2V}+\frac{\left(a-\overline{a}\right)^{2}\left(e^{{-V}}+e^{V}\right)\sin \frac{a+\overline{a}}{2}}{8V^{2}}\right. \\ &\left.+\frac{\left(a-\overline{a}\right)\left(e^{{-V}}-e^{V}\right)\cos \frac{a+\overline{a}}{2}}{4V}+\frac{\left(a-\overline{a}\right)^{2}\left(e^{{-V}}-e^{V}\right)\sin \frac{a+\overline{a}}{2}}{8V^{3}}\right]\!,
\end{split} \]
\[ \begin{split}
\partial_{\overline{a}}\Phi_{{1\left(\cos p\right)}}=\partial_{\overline{a}}&\left[\frac{\left(e^{{-V}}+e^{V}\right)\cos \frac{a+\overline{a}}{2}}{2}+\frac{\left(a-\overline{a}\right)\left(e^{{-V}}-e^{V}\right)\sin \frac{a+\overline{a}}{2}}{{4V}}\right]\\
  = \frac{1}{2}&\left[-\frac{\left(a-\overline{a}\right)\left(e^{{-V}}-e^{V}\right)\cos \frac{a+\overline{a}}{2}}{{4V}}-\frac{\left(e^{{-V}}+e^{V}\right)\sin \frac{a+\overline{a}}{2}}{2}\right]\\
+\frac{1}{2}&\left[-\frac{\left(e^{{-V}}-e^{V}\right)\sin \frac{a+\overline{a}}{2}}{2V}-\frac{\left(a-\overline{a}\right)^{2}\left(e^{{-V}}+e^{V}\right)\sin \frac{a+\overline{a}}{2}}{8V^{2}}\right.  \\ &\left.+\frac{\left(a-\overline{a}\right)\left(e^{{-V}}-e^{V}\right)\cos \frac{a+\overline{a}}{2}}{4V}-\frac{\left(a-\overline{a}\right)^{2}\left(e^{{-V}}-e^{V}\right)\sin \frac{a+\overline{a}}{2}}{8V^{3}}\right]\!.\\
\end{split} \]

Substituting these expressions into (\ref{Eq38}) and reducing like terms, we get
\begin{equation}  \label{Eq40}   
\Phi_{{1\left(\cos p\right)}}^{'}=-\frac{\left(e^{{-V}}+e^{V}\right)\sin \frac{a+\overline{a}}{2}}{2}+\frac{\left(a-\overline{a}\right)\left(e^{{-V}}-e^{V}\right)\cos \frac{a+\overline{a}}{2}}{{4V}}=-\Phi_{{1\left(\sin p\right)}}.
\end{equation}

Quite analogously we obtain for the partial derivatives $\partial_{a}\Phi_{{2\left(\cos p\right)}}$ and $\partial_{\overline{a}}\Phi_{{2\left(\cos p\right)}}$ the following expressions:
\begin{equation*}
 \begin{split}
\partial_{a}\Phi_{{2\left(\cos p\right)}}&=\frac{b}{2}\,\partial_{a}\left[\frac{\left(e^{{-V}}-e^{V}\right)\sin \frac{a+\overline{a}}{2}}{V}\right]=\frac{b}{2}\left[\frac{\left(e^{{-V}}-e^{V}\right)\cos \frac{a+\overline{a}}{2}}{2V}\right. \\
&  +\left.\frac{\left(a-\overline{a}\right)\left(e^{{-V}}+e^{V}\right)\sin \frac{a+\overline{a}}{2}}{{4V^{2}}}+\frac{\left(a-\overline{a}\right)\left(e^{{-V}}-e^{V}\right)\sin \frac{a+\overline{a}}{2}}{{4V^{3}}}\right],\\
\end{split}
\end{equation*}
\begin{equation*}
 \begin{split}
\partial_{\overline{a}}\Phi_{{2\left(\cos p\right)}}&=\frac{b}{2}\,\partial_{\overline{a}}\left[\frac{\left(e^{{-V}}-e^{V}\right)\sin \frac{a+\overline{a}}{2}}{V}\right]=\frac{b}{2}\left[\frac{\left(e^{{-V}}-e^{V}\right)\cos \frac{a+\overline{a}}{2}}{2V}\right. \\
&  -\left.\frac{\left(a-\overline{a}\right)\left(e^{{-V}}+e^{V}\right)\sin \frac{a+\overline{a}}{2}}{{4V^{2}}}-\frac{\left(a-\overline{a}\right)\left(e^{{-V}}-e^{V}\right)\sin \frac{a+\overline{a}}{2}}{{4V^{3}}}\right].\\
\end{split}
\end{equation*}
Substituting these expressions into (\ref{Eq39}) we get the second constituent of the full quaternionic derivative of $\cos p$:
\begin{equation} \label{Eq41} 
\Phi_{{2\left(\cos p\right)}}^{'}=\partial_{a}\Phi_{{2\left(\cos p\right)}}+\partial_{\overline{a}}\Phi_{{2\left(\cos p\right)}}=\frac{\left(e^{{-V}}-e^{V}\right)\cos \frac{a+\overline{a}}{2}}{2V}b=-\Phi_{{2\left(\sin p\right)}}.
\end{equation}

Next, substituting (\ref{Eq40}) and (\ref{Eq41}) into (\ref{Eq37a}), we get the following expression:
\begin{equation} \label{Eq42}
\left(\cos p\right)^{'}=\Phi_{{1\left(\cos p\right)}}^{'}+\Phi_{{2\left(\cos p\right)}}^{'} \cdot j=-\Phi_{{1\left(\sin p\right)}}-\Phi_{{2\left(\sin p\right)}} \cdot j=-\sin p.
\end{equation}

Thus, we also have shown that the derivatives of the quaternionic and complex cosine functions have indeed identical expressions in accordance with Corollary \ref{co1.4}.

In an analogous fashion, we obtain for the quaternionic derivative of the function $\sin p$ the following expression:
\begin{equation} \label{Eq43} 
\left(\sin p\right)^{'}=\cos p,
\end{equation}
which coincides with the corresponding expression from the complex analysis. We skip the details here.

 Formulae (\ref{Eq42}) and (\ref{Eq43}) allow us further to calculate, using (\ref{Eq22}) , the coefficients $r_{k}=\frac{\psi_H^{\left(k\right)}\!\!\left(0\right)}{k\,!},\,\,\,\,k=0,1,2,\dots,$ of Maclaurin series expansions for the functions $\sin p$ and $\cos p$ to show that they coincide with corresponding coeffitients in  (\ref{Eq24}) and (\ref{Eq25}). 

To illustrate this we consider the function $\sin p.$ For $\sin p$ we get the coefficient of $p^{1}$ as follows: $r_{1}=\frac{\psi_H^{\left(1\right)}\!\!\left(0\right)}{1\,!}=\left(\sin p\right)_{p=0}^{{\left(1\right)}}=\cos 0=1.$ It coincides with the coefficient of $p^{1}$ in (\ref{Eq24}) that equals 1. The calculation of the coefficient of $p^{2}$ gives $r_{2\left(\sin p\right)}=\frac{\psi_H^{\left(2\right)}\!\!\left(0\right)}{2!}=\frac{\left(\sin p\right)_{p=0}^{\left(2\right)}}{2}=\frac{\left(\cos p\right)_{p=0}^{\left(1\right)}}{2}=-\frac{\sin 0}{2}=0.$ It coincides with the coefficient of $p^{2}$ in (\ref{Eq24}) that is equal to 0. The coefficient of $p^{3}$ is calculated as follows: $r_{3\left(\sin p\right)}=\frac{\psi_H^{\left(3\right)}\!\!\left(0\right)}{3!}=\frac{\left(\sin p\right)_{p=0}^{\left(3\right)}}{3!}=-\frac{\cos 0}{3!}=-\frac{1}{3!}.$ It coincides with the coefficient of $p^{3}$ in (\ref{Eq24}) that is equal to $-\frac{1}{{3!}}.$ In this way, we can go on. 

We have shown that   the application of Theorems \ref{th1.2}, \ref{th1.3}, and Corollary \ref{co1.4} to infinite series gives true results.

\subsection{Example of a product of H-holomorphic functions}  \label{sub4.2}
Consider  a more complicated example: a quaternionic product of two H-holomorphic functions $\sin p$ and $\cos p$. Since a qiaternionic product of H-holomorphic functions is also a H-holomorphic function (see \!\cite[\!p.\!\!~18]{pm:eac}), we can, using (\ref{Eq6}), calculate the full quaternionic derivative of each order and then, using (\ref{Eq22}), obtain each coefficient $r_{k}$ of a Maclaurin series expansion for the quaternionic product in question.  However, such a calculation is cumbersome and we can use in accordance with Corollary \ref{co1.4} the following product differentiation rule   (see also \cite[\!p.\!~20]{pm:eac}): ${{\left(f\cdot g\right)}}^{'}=f^{'}\cdot g+f\cdot g^{'},$ where $f$ and $g$ are arbitrary H-holomorphic functions.

To establish a general expression for terms of this series we  calculate a sufficient number of the terms and  derivatives. 
Given  (\ref{Eq42}) and (\ref{Eq43}),  we obtain for the derivatives of the function $\psi_{H}\!\left(p\right)=\sin p \cdot \cos p$ the following expressions:
\begin{align*} 
&\psi_H^{{\left(1\right)}}\!\left(p\right)=\cos^{2}p-\sin^{2}p, \,\,\,\,\psi_H^{{\left(2\right)}}\!\left(p\right)=-4\sin p \,\cos p,\,\,\,\,\psi_H^{{\left(3\right)}}\!\left(p\right)=-4\left(\cos^{2}p-\sin^{2}p\right),\\  & \psi_H^{{\left(4\right)}}\!\left(p\right)=16\sin p \,\cos p,\,\,\,\,\psi_H^{{\left(5\right)}}\!\left(p\right)=16\left(\cos^{2}p-\sin^{2}p\right),\,\,\,\,\psi_H^{{\left(6\right)}}\!\left(p\right)=-64\sin p \,\cos p ,\\ &  \psi_H^{{\left(7\right)}}\!\left(p\right)=-64\left(\cos^{2}p-\sin^{2}p\right)\!,\,\psi_H^{{\left(8\right)}}\!\left(p\right)=256\sin p \,\cos p,\,\psi_H^{{\left(9\right)}}\!\left(p\right)=256\left(\cos^{2}p-\sin^{2}p\right)\!,\\ & \psi_H^{{\left(10\right)}}\!\left(p\right)=-1024\sin p \,\cos p ,\,\,\,\psi_H^{{\left(11\right)}}\!\left(p\right)=-1024\left(\cos^{2}p-\sin^{2}p\right)\!,\,\\&\psi_H^{{\left(12\right)}}\!\left(p\right)=4096\sin p \,\cos p,\,\,\,
\psi_H^{{\left(13\right)}}\!\left(p\right)=4096\left(\cos^{2}p-\sin^{2}p\right),\,\,\,\,\\ &\psi_H^{{\left(14\right)}}\!\left(p\right)=-16384\sin p \,\cos p,\,\,\,\,\psi_H^{{\left(15\right)}}\!\left(p\right)=-16384\left(\cos^{2}p-\sin^{2}p\right), \\&\psi_H^{{\left(16\right)}}\!\left(p\right)=65536\sin p \,\cos p ,\,\,\,\, \psi_H^{{\left(17\right)}}\!\left(p\right)=65536\left(\cos^{2}p-\sin^{2}p\right).
\end{align*}

Substituting $p=0$ into these expressions and using (\ref{Eq22}), we get the coeffitients  of the  Maclaurin series expansion of the function $\psi_{H}\!\left(p\right)=\sin p \cdot \cos p$ as follows: \nopagebreak[0]
\begin{align*}  
&r_{0}=\frac{\psi_{H}\!\left(0\right) }{{0!}}=\frac{0}{{1}}=0,\,\,r_{1}=\frac{\psi_H^{{\left(1\right)}}\left(0\right)}{{1!}}=\frac{\left(\cos^{2}0-\sin^{2}0\right)}{{1!}}=\frac{1}{{1!}}=1,\,\,r_{2}=\frac{\psi_H^{{\left(2\right)}}\left(0\right)}{{2!}}=0,\\&r_{3}=\frac{\psi_H^{{\left(3\right)}}\left(0\right)}{{3!}}=-\frac{4}{{3!}},\,\,r_{4}=\frac{\psi_H^{{\left(4\right)}}\left(0\right)}{{4!}}=0,\,\,r_{5}=\frac{\psi_H^{{\left(5\right)}}\left(0\right)}{{5!}}=\frac{16}{{5!}},\,\,r_{6}=\frac{\psi_H^{{\left(6\right)}}\left(0\right)}{{6!}}=0, \\ &  r_{7}=\frac{\psi_H^{{\left(7\right)}}\left(0\right)}{{7!}}=-\frac{64}{{7!}},\,\,r_{8}=\frac{\psi_H^{{\left(8\right)}}\left(0\right)}{{8!}}=0,\,\,r_{9}=\frac{\psi_H^{{\left(9\right)}}\left(0\right)}{{9!}}=\frac{256}{{9!}},\,\,r_{10}=\frac{\psi_H^{{\left(10\right)}}\left(0\right)}{{10!}}=0,\\ &r_{11}=\frac{\psi_H^{{\left(11\right)}}\left(0\right)}{{11!}}=-\frac{1024}{{11!}},\,\,r_{12}=\frac{\psi_H^{{\left(12\right)}}\left(0\right)}{{12!}}=0,\,\,
r_{13}=\frac{\psi_H^{{\left(13\right)}}\left(0\right)}{{13!}}=\frac{4096}{{13!}},\\&r_{14}=\frac{\psi_H^{{\left(14\right)}}\left(0\right)}{{14!}}=0,\,\,r_{15}=\frac{\psi_H^{{\left(15\right)}}\left(0\right)}{{15!}}=-\frac{16384}{{15!}},\,\, r_{16}=\frac{\psi_H^{{\left(16\right)}}\left(0\right)}{{16!}}=0,\\
 &r_{17}=\frac{\psi_H^{{\left(17\right)}}\left(0\right)}{{17!}}=\frac{65536}{{17!}}.
\end{align*}

Thus, we get the Maclaurin series expansion for the function $\psi_{H}\!\left(p\right)=\sin p \cdot \cos p$ as follows: 
\begin{align*} 
\sin p\cdot \cos p&=p-\frac{4}{{3!}}p^{3}+\frac{16}{{5!}}p^{5}-\frac{64}{{7!}}p^{7}
+\frac{256}{{9!}}p^{9}\\
 &-\frac{1024}{{11!}}p^{11}+\frac{4096}{{13!}}p^{13}-\frac{16384}{{15!}}p^{15}+\frac{65536}{{17!}}p^{17}-\cdots
\end{align*} 
  
In a general form this series is the following:
\begin{equation} \label{Eq46}
\sin p \cdot \cos= \sum_{l=0}^{\infty}\alpha_{l},
\end{equation}
where
\begin{equation} \label{Eq46a}
\alpha_{l}={{\left(-1\right)}}^{l}\frac{4^{l}p^{{2l+1}}}{{\left(2l+1\right)!}}.
\end{equation}

To verify a convergence  of series (\ref{Eq46}) we use d'Alembert's Ratio Test (see Proposition \ref{pr3.5}). Replacing $l$ by $l+1$ in (\ref{Eq46a}) we get the following expression:
\begin{equation} \label{Eq49}
\alpha_{{l+1}}={{\left(-1\right)}}^{{l+1}}\frac{4^{{l+1}}p^{{2l+3}}}{{\left(2l+3\right)!}}.
\end{equation}

Now, using (\ref{Eq46a}) and (\ref{Eq49}) and reducing like terms, we get the quotient of $\alpha_{{l+1}}$ by $\alpha_{{l}}$: 
\begin{equation} \label{Eq50}
\frac{\alpha_{{l+1}}}{\alpha_{l}}=\frac{{{\left(-1\right)}}^{{l+1}}\frac{4^{{l+1}}p^{{2l+3}}}{{\left(2l+3\right)!}}}{{{\left(-1\right)}}^{l}\frac{4^{l}p^{{2l+1}}}{{\left(2l+1\right)!}}}=-\frac{4p^{2}}{\left(2l+2\right)\left(2l+3\right)}.
\end{equation}
 
 Further, taking limit  (\ref{Eq18d}), we have as follows:
\begin{equation} 
 \lim_{l \rightarrow \infty}=\frac{\mid\!\alpha_{{l+1}}\!\mid}{\mid\!\alpha_{l}\!\mid}=L= \lim_{l \rightarrow \infty}\frac{4\mid \!p^{2}\!\mid}{\left(2l+2\right)\left(2l+3\right)}=0,
\end{equation}
i.e. the  series (\ref{Eq46}) is absolutely convergent at all $p\in\mathbb{H}.$ 

To check whether  series (\ref{Eq46}) is uniformly convergent at all $p\in\mathbb{H}$ we can use the value $M_{l}=\frac{5^{l}p^{{2l+1}}}{{\left(2l+1\right)!}}$ such that $\mid \! \! \alpha_{l} \! \! \mid=\frac{4^{l}p^{{2l+1}}}{{\left(2l+1\right)!}}\leq M_{l}$ for $l=0,1,2,\dots$ and at all $p\in\mathbb{H}.$ At that, the series $\sum_{l=0}^{\infty}M_{l}$ is absolutely convergent at all $p\in\mathbb{H}$  in accordance with Proposition \ref{pr3.5}, since $\lim_{l \rightarrow \infty}\frac{\mid M_{{l+1}}\mid}{\mid M_{l}\mid}=\lim_{l \rightarrow \infty}\frac{5\mid p\mid^{2}}{{\left(2l+2\right)\left(2l+3\right)}}=0.$ 

Thus, according to Proposition \ref{pr3.6}, series (\ref{Eq46}) converges absolutely and  uniformly at all  $p\in\mathbb{H}$. 

Bearing in mind the paramount importance of the isomorphism between the C- and H-representations for the theory in question, we would also like to demonstrate by a direct calculation that the quaternionic multiplication of the H-holomorphic functions $\sin p$ and $\cos p$ behaves as commutative. In the case of multiplication of two arbitrary quaternionic functions $f\!\left(p\right)=f_{1}+f_{2}\cdot j$ and  $g\!\left(p\right)=g_{1}+g_{2}\cdot j$ formula (\ref{Eq12a}) becomes the following:  
 \begin{equation} \label{Eq52}
f\cdot g=\left(f_{1}+f_{2}\cdot j\right)\cdot\left(g_{1}+g_{2}\cdot j\right) =Re\left(f\cdot g\right)+Im\left(f\cdot g\right)\cdot j,
\end{equation}
where $$Re\left(f\cdot g\right)=f_{1}\,g_{1}-f_{2}\,\overline{g_{2}},\,\,\,\,\, Im\left(f\cdot g\right)=f_{2}\,\overline{g_{1}}+f_{1}\,g_{2}$$\\
are the designations of ''real'' and ``imaginary'' parts of the quaternionic product. 

\textit{The quaternionic product} $\sin p\cdot\cos p.$ Using (\ref{Eq28}), (\ref{Eq29}), (\ref{Eq30}), and (\ref{Eq31}) we have in this case the following expressions:
\begin{align*} 
&f_{1}=\Phi_{{1\left(\sin p\right)}}=\frac{\left(e^{{-V}}+e^{V}\right)\sin \frac{a+\overline{a}}{2}}{2}-\frac{\left(a-\overline{a}\right)\left(e^{{-V}}-e^{V}\right)\cos \frac{a+\overline{a}}{2}}{{4V}}, \\&f_{2}=\Phi_{{2\left(\sin p\right)}}=-\frac{\left(e^{{-V}}-e^{V}\right)\cos \frac{a+\overline{a}}{2}}{{2V}}b,\\  &g_{1}=\Phi_{{1\left(\cos p\right)}}=\frac{\left(e^{{-V}}+e^{V}\right)\cos \frac{a+\overline{a}}{2}}{2}+\frac{\left(a-\overline{a}\right)\left(e^{{-V}}-e^{V}\right)\sin \frac{a+\overline{a}}{2}}{{4V}},\\ & g_{2}=\Phi_{{2\left(\cos p\right)}}=\frac{\left(e^{{-V}}-e^{V}\right)\sin \frac{a+\overline{a}}{2}}{{2V}}b.
\end{align*}
Substituting these expressions into (\ref{Eq52}) and taking into account that $V=\overline{V}$ and $\beta=\overline{\beta}$ we obtain as follows:
\begin{equation} \label{Eq53} 
 \begin{gathered}
Re\left(\sin p \cdot \cos p\right)=f_{1}g_{1}-f_{2}\overline{g_{2}}\\
=\left[\frac{\left(e^{{-V}}+e^{V}\right)\sin\frac{a+\overline{a}}{2}}{2}
-\frac{\left(a-\overline{a}\right)\left(e^{{-V}}-e^{V}\right)\cos \frac{a+\overline{a}}{2}}{{4V}}\right]\left[\frac{\left(e^{{-V}}+e^{V}\right)\cos\frac{a+\overline{a}}{2}}{2}\right.\\
+\left. \frac{\left(a-\overline{a}\right)\left(e^{{-V}}-e^{V}\right)\sin \frac{a+\overline{a}}{2}}{{4V}}\right]+\frac{\left(e^{{-V}}-e^{V}\right)\cos\frac{a+\overline{a}}{2}}{2V}\,b\,\frac{\left(e^{{-V}}-e^{V}\right)\sin\frac{a+\overline{a}}{2}}{2V}\,\overline{b},
\end{gathered}
\end{equation}
\begin{equation} \label{Eq54} 
 \begin{gathered}
Im\left(\sin p \cdot \cos p\right)=f_{2}\overline{g_{1}}+f_{1}g_{2}\\
=-\frac{\left(e^{{-V}}-e^{V}\right)\cos\frac{a+\overline{a}}{2}}{2V}\,b\,\left[\frac{\left(e^{{-V}}+e^{V}\right)\cos\frac{a+\overline{a}}{2}}{2}
+\frac{\left(a-\overline{a}\right)\left(e^{{-V}}-e^{V}\right)\sin \frac{a+\overline{a}}{2}}{{4V}}\right]\\
+\left[\frac{\left(e^{{-V}}+e^{V}\right)\sin\frac{a+\overline{a}}{2}}{2}
-\frac{\left(a-\overline{a}\right)\left(e^{{-V}}-e^{V}\right)\cos \frac{a+\overline{a}}{2}}{{4V}}\right]\frac{\left(e^{{-V}}-e^{V}\right)\sin\frac{a+\overline{a}}{2}}{2V}\,b \\ =\frac{\left(e^{{-2V}}-e^{2V}\right)}{4V}\,b\,\left(-\cos^{2}\frac{a+\overline{a}}{2}+\sin^{2}\frac{a+\overline{a}}{2}\right).
\end{gathered}
\end{equation}

\textit{The quaternionic product} $\cos p \cdot \sin p.$ In this case we must substitute into (\ref{Eq52}) the following expressions:
\begin{align*} 
f_{1}&=\Phi_{{1\left(\cos p\right)}}=\frac{\left(e^{{-V}}+e^{V}\right)\cos \frac{a+\overline{a}}{2}}{2}+\frac{\left(a-\overline{a}\right)\left(e^{{-V}}-e^{V}\right)\sin \frac{a+\overline{a}}{2}}{{4V}},\\   f_{2}&=\Phi_{{2\left(\cos p\right)}}=\frac{\left(e^{{-V}}-e^{V}\right)\sin \frac{a+\overline{a}}{2}}{{2V}}b,\\
g_{1}&=\Phi_{{1\left(\sin p\right)}}=\frac{\left(e^{{-V}}+e^{V}\right)\sin \frac{a+\overline{a}}{2}}{2}-\frac{\left(a-\overline{a}\right)\left(e^{{-V}}-e^{V}\right)\cos \frac{a+\overline{a}}{2}}{{4V}},\\ g_{2}&=\Phi_{{2\left(\sin p\right)}}=-\frac{\left(e^{{-V}}-e^{V}\right)\cos \frac{a+\overline{a}}{2}}{{2V}}b.
\end{align*}

After substitution we have as follows:
\begin{equation} \label{Eq55} 
 \begin{gathered}
Re\left( \cos p \cdot \sin p \right)=f_{1}g_{1}-f_{2}\overline{g_{2}}\\
=\left[\frac{\left(e^{{-V}}+e^{V}\right)\cos\frac{a+\overline{a}}{2}}{2}
+\frac{\left(a-\overline{a}\right)\left(e^{{-V}}-e^{V}\right)\sin \frac{a+\overline{a}}{2}}{{4V}}\right]\left[\frac{\left(e^{{-V}}+e^{V}\right)\sin\frac{a+\overline{a}}{2}}{2}\right.\\
-\left. \frac{\left(a-\overline{a}\right)\left(e^{{-V}}-e^{V}\right)\cos \frac{a+\overline{a}}{2}}{{4V}}\right]+\frac{\left(e^{{-V}}-e^{V}\right)\sin\frac{a+\overline{a}}{2}}{2V}\,b\,\frac{\left(e^{{-V}}-e^{V}\right)\cos\frac{a+\overline{a}}{2}}{2V}\overline{b},
\end{gathered}
\end{equation}
\begin{equation} \label{Eq56} 
 \begin{gathered}
Im \left( \cos p \cdot \sin p \right)=f_{2}\overline{{g_{1}}}+f_{1}g_{2}\\
=\frac{\left(e^{{-V}}-e^{V}\right)\sin\frac{a+\overline{a}}{2}}{2V}\,b\,\left[\frac{\left(e^{{-V}}+e^{V}\right)\sin\frac{a+\overline{a}}{2}}{2}
+\frac{\left(a-\overline{a}\right)\left(e^{{-V}}-e^{V}\right)\cos \frac{a+\overline{a}}{2}}{{4V}}\right]\\
-\left[\frac{\left(e^{{-V}}+e^{V}\right)\cos\frac{a+\overline{a}}{2}}{2}+\frac{\left(a-\overline{a}\right)\left(e^{{-V}}-e^{V}\right)\sin \frac{a+\overline{a}}{2}}{{4V}}\right]\frac{\left(e^{{-V}}-e^{V}\right)\cos\frac{a+\overline{a}}{2}}{2V}\,b\\
=\frac{\left(e^{{-2V}}-e^{2V}\right)}{4V}\,b\,\left(\sin^{2}\frac{a+\overline{a}}{2}-\cos^{2}\frac{a+\overline{a}}{2}\right).
\end{gathered}
\end{equation}

Comparing (\ref{Eq53}) and (\ref{Eq55}) as well as (\ref{Eq54}) and (\ref{Eq56}) we see that
\begin{equation*} 
Re\left(\sin p \cdot \cos p\right)=Re\left(\cos p \cdot \sin p\right),\,\,\,\,\,\,Im\left(\sin p\cdot \cos p\right)=Im\left(\cos p \cdot \sin p\right).
\end{equation*}

Thus we have shown that the quaternionic multiplication of H-holomorphic functions $\sin p$ and $\cos p$ behaves as commutative.

\hspace{20mm}

In conclusion, it is not superfluous to note  that the theory of H-holomorphic functions in question possesses, in a sense, the so-called "internal perfection", since this theory follows in essence from one main idea  \cite[\!p.\!\!~15]{pm:eac}: "each point of any real line is at the same time a point of some plane and space as a whole, and therefore any characterization of differentiability at a point must be the same, regardless of whether we think of that point as a point on the real axis or a point in the complex plane, or a point in three-dimensional space". It leads to introducing the conception of one holistic  notion of a holomorphic function that in the case of complex variables has a C-representation and in the case of  quaternions has a corresponding H-representation.  Moreover,  such a single vision is also confirmed by the studied quaternionic power series, which  in C- and H-representations  of any holomorphic function are similar and converge  with identical convergence radiuses. Analogously, decompositions of C-analytic and H-analytic functions have the same radiuses of convergence. We also can speak of the so-called "external justification" of the theory in question, since this theory  is, as far as we know, the only theory, which explains the fact that quaternionic multiplication of  H-holomorphic functions behaves as commutative. This is a reasonable ground to approve that the presented theory of H-holomorphic functions is true. 
 
  \label{sec:intro}
\bibliography{bibliography}
\end{document}